\documentclass[a4paper,12pt]{article}
\usepackage{amsmath,amsthm,amssymb}
\usepackage{mathptmx}
\usepackage{graphicx,graphics}
\usepackage{tikz}
\usepackage{hyperref}
\newtheorem{theorem}{Theorem}[section]
\newtheorem{corollary}[theorem]{Corollary}
\newtheorem{lemma}[theorem]{Lemma}

\theoremstyle{definition}
\newtheorem{definition}[theorem]{Definition}

\usepackage{mathrsfs}
\usepackage{enumerate}
\usepackage[english]{babel}
\usepackage{graphicx}
\usetikzlibrary{shapes.geometric, arrows}
\usetikzlibrary{mindmap}
\usepackage{ccaption, amsthm}
\captionnamefont{\scshape}
\captionstyle{}
\usepackage{caption}
\DeclareCaptionSubType*{figure}
\usepackage{subfig}
\usepackage{floatrow}
\usepackage{placeins} % optional, for \FloatBarrier
\usepackage{enumitem}
\usepackage{graphicx}
\usepackage{tikz}
\usetikzlibrary{shapes.geometric,arrows}
\usetikzlibrary{mindmap}
\usetikzlibrary{automata}
\usetikzlibrary{positioning,calc}
\usetikzlibrary{graphs}
\usetikzlibrary{graphs.standard}
\usetikzlibrary{arrows,decorations.markings}
\usepackage{tkz-graph}
\usetikzlibrary{chains,fit,shapes}
\usetikzlibrary{calc}
\usepackage{algorithm}%
\usepackage{algorithmicx}%
\usepackage{algpseudocode}%
\usepackage{listings}%
\numberwithin{equation}{section}
\usepackage{tikz}
\usetikzlibrary{automata}
\usetikzlibrary{arrows}
\usetikzlibrary{positioning,calc}
\usetikzlibrary{graphs}
\usetikzlibrary{graphs.standard}
\usetikzlibrary{arrows,decorations.markings}
\usepackage{tkz-graph}
\usetikzlibrary{chains,fit,shapes}
\usetikzlibrary{calc}
\newcount\nodecount
\tikzgraphsset{
	declare={subgraph N}%
	{
		[/utils/exec={\global\nodecount=0}]
		\foreach \nodetext in \tikzgraphV
		{  [/utils/exec={\global\advance\nodecount by1}, 
			parse/.expand once={\the\nodecount/\nodetext}] }
	},
	declare={subgraph C}%
	{
		[cycle, /utils/exec={\global\nodecount=0}]
		\foreach \nodetext in \tikzgraphV
		{  [/utils/exec={\global\advance\nodecount by1}, 
			parse/.expand once={\the\nodecount/\nodetext}] }
	},
	declare={subgraph W}%
	{
		[/utils/exec={\global\nodecount=0}]
		\foreach \nodetext in \tikzgraphV
		{  [/utils/exec={\global\advance\nodecount by1}, 
			parse/.expand once={\the\nodecount/\nodetext}] }
	}
}
\definecolor{qqccqq}{rgb}{0,0.8,0}
\definecolor{qqzzcc}{rgb}{0,0.6,0.8}
\tikzset{every loop/.style={min distance=10mm,looseness=10}}
\tikzset{every state/.style={minimum size=2mm}}

\newcommand\MI[1]{M_{\mathit{I}}(#1)}
\newcommand\MII[1]{M_{\mathit{II}}(#1)}
\newcommand\MIII[1]{M_{\mathit{III}}(#1)}
\newcommand\MIV{M_{\mathit{IV}}}
\newcommand\MV{M_{\mathit{V}}}
\newcommand\MVI{M_{\mathit{VI}}}
\newcommand\MVII{M_{\mathit{VII}}}

\newcommand\MIast[1]{M_{\mathit{I}}^*(#1)}

\newcommand\MVast{M_{\mathit{V}}^*}

\newcommand\ForbRow{\mathcal F_{\textup{circR}}}

\newcommand\IForbRow{\mathcal F_{\textup{IcircR}}}
\newcommand\GForbRow{\mathcal G_{\textup{split}}}

\newcommand\miop{\mathbin{\oplus}}

\makeatletter
\let\c@subfigure\relax
\makeatother

\makeatletter
	\title{Forbidden Induced Subgraph Characterization of Word-Representable Split Graphs}
\author{\hspace{1cm} Eshwar Srinivasan \ and Ramesh Hariharasubramanian \\
	{{\footnotesize s.eshwar@iitg.ac.in},\ {\footnotesize  ramesh\_h@iitg.ac.in}}\\{\footnotesize Department of Mathematics, Indian Institute of Technology Guwahati, Guwahati, Assam 781039, India}}
\begin{document}
	
	\maketitle
	\begin{abstract}
		The class of word-representable graphs, introduced in connection with the study of the Perkins semigroup by Kitaev and Seif, has attracted significant attention in combinatorics and theoretical computer science due to its deep connections with graph orientations and combinatorics on words. A graph is word-representable if and only if it admits a semi-transitive orientation, which is an acyclic orientation such that for any directed path \(v_0 \rightarrow v_1 \rightarrow \cdots \rightarrow v_m\) with \(m \ge 2\), either there is no arc between \(v_0\) and \(v_m\), or, for all \(1 \le i < j \le m\), there exists an arc from \(v_i\) to \(v_j\). Split graphs, whose vertex set can be partitioned into a clique and an independent set, constitute a natural yet nontrivial subclass for studying word-representability. However, not all split graphs are semi-transitive, and the characterization of minimal forbidden induced subgraphs for semi-transitive split graphs remains an open problem.
		
		In this paper, we introduce a new matrix property called the \(I\)-circular property, which is closely related to the well-known \(D\)-circular property introduced by Safe. The \(I\)-circular property requires that both the rows of a matrix and the pairwise intersections of rows form circular intervals under some linear ordering of the columns. Using this property, we establish a direct connection between the structure of semi-transitive split graphs and the matrix representation of their adjacency relationships. Our main result is a complete forbidden submatrix characterization of the \(I\)-circular property, which in turn provides a characterization for semi-transitive split graphs in terms of minimal forbidden induced subgraphs.
		
	\end{abstract}

	\section{Introduction}
	
	A graph \(G = (V,E)\) is said to be \emph{word-representable} if there exists a word \(w\) over the alphabet \(V\) such that for any two distinct letters \(x,y \in V\), the letters \(x\) and \(y\) alternate in \(w\) if and only if \((x,y)\) is an edge in \(E\). This concept, which first emerged in the study of the Perkins semigroup through the work of Kitaev and Seif~\cite{kitaev2008word}, has generated significant interest and inspired a broad range of research focused on the recognition and characterization of such graphs.
	
	A crucial characterization, due to Halldórsson et al.~\cite{HALLDORSSON2016164}, establishes \emph{semi-transitive orientations} as the defining structural property of word-representable graphs. That is, a graph is word-representable if and only if it admits a semi-transitive orientation. A \textit{semi-transitive} orientation is an acyclic orientation such that, for any directed path $v_0 \rightarrow v_1 \rightarrow \cdots \rightarrow v_m$ with $m \ge 2$, either there is no arc between $v_0$ and $v_m$, or, for all $1 \le i < j \le m$, an arc exists from $v_i$ to $v_j$. This elegant characterization places semi-transitive orientations at the core of the theory of word-representable graphs, serving as a conceptual bridge between combinatorics on words and graph structure. For further reading on word-representable graphs, we refer the reader to~\cite{kitaev2008word, kitaev2008representable, kitaev2011representability, kitaev13, HALLDORSSON2016164, kitaev2015words, kitaev2017comprehensive, broere2018word, kitaev2021human, srinivasan2024minimum, huang2024embedding}.
	
	In this broader context, \emph{split graphs}—graphs whose vertex set can be partitioned into a clique and an independent set—constitute an important subclass for studying word-representability. The study of word-representable split graphs is motivated by both theoretical and practical considerations. From a theoretical perspective, split graphs occupy a significant place in graph theory: they are hereditary, meaning that any induced subgraph of a split graph is itself a split graph, and they form a subclass of chordal graphs. Their simple definition, combined with rich structural properties, makes them a natural setting for testing conjectures and exploring graph-theoretic phenomena. From a practical perspective, split graphs serve as models for real-world systems comprising a densely connected core interacting with a sparsely connected periphery. Examples include social networks, where a close-knit group interacts with occasional outsiders; biological networks; and database systems designed for hierarchical queries. These features make split graphs a valuable object of study in combinatorics and applications.
	
	A \emph{hereditary graph class} can often be characterized by a family of \emph{forbidden induced subgraphs}: that is, a graph belongs to the class if and only if it contains no graph from a specified family \(\mathcal{F}\) as an induced subgraph. The set \(\mathcal{F}\) is called the collection of \emph{minimal forbidden induced subgraphs} for the class. For example, \emph{cographs} are precisely the graphs that do not contain an induced path on four vertices (\(P_4\))~\cite{corneil1981}. Similarly, \emph{trivially perfect graphs} are those that forbid both the path \(P_4\) and the cycle \(C_4\)~\cite{golumbic1980}.
	
	The study of hereditary graph classes and their forbidden subgraph characterizations is central to algorithm design and complexity theory. When the set of minimal forbidden induced subgraphs \(\mathcal{F}\) is finite, testing membership in the class reduces to checking for the presence of finitely many induced subgraphs, often enabling the design of efficient, polynomial-time recognition algorithms. Not every hereditary graph class, however, admits such a finite forbidden subgraph characterization, which poses additional challenges for both structural analysis and algorithm development.
	
	Since word-representable graphs are hereditary, a line of research has focused on identifying their set of minimal forbidden induced subgraphs, which remains unknown. As a result, the problem of determining the forbidden induced subgraphs for word-representable graphs can be approached by restricting attention to smaller graph classes, such as split graphs.
	
     The word-representability of split graphs was first studied by Kitaev et al.~\cite{kitaev2021word}, where the authors characterized word-representable split graphs in terms of forbidden subgraphs in which vertices in the independent set have degree at most $2$, or the size of the clique is $4$. Moreover, they provided necessary and sufficient conditions for an orientation of a split graph to be semi-transitive. Subsequently, in~\cite{chen2022representing}, Chen et al. computationally characterized word-representable split graphs with clique size $5$ in terms of forbidden subgraphs. More recently, in~\cite{roy2025word}, Roy and one of the present authors characterized word-representable split graphs with independent set size $4$, in terms of forbidden subgraphs as well.
		
	Recent work by Kitaev and Pyatkin~\cite{kitaev2024semi} made significant progress by characterizing semi-transitive split graphs using matrix-theoretic tools. They showed that the adjacency structure of a split graph can be encoded as a \((0,1)\)-matrix and analyzed via the \emph{circular-ones property}. Specifically, they proved that a split graph is semi-transitive if and only if its adjacency matrix satisfies the circular-ones property for rows, together with an additional structural constraint on the placement of ones.
	
	Among various matrix properties, the circular-ones property has found applications in several graph-theoretic characterizations. The circular-ones property (see \textit{Definition}~\ref{cir}) generalizes the consecutive-ones property. In~\cite{tucker1972structure}, Tucker provided a forbidden submatrix characterization for the consecutive-ones property. Later, Booth and Lueker~\cite{booth1976pq} presented a linear-time algorithm for efficiently recognizing the consecutive-ones property. In~\cite{mdsafe1}, Safe provided an analogous characterization of the circular-ones property and presented a linear-time algorithm to recognize it. Furthermore, in~\cite{mdsafe2}, Safe introduced a matrix property called the \emph{D-circular property}, a special case of the circular-ones property in which both the rows and the set difference between any two rows form circular intervals (see \textit{Definition }\ref{ciri}) under some linear ordering of columns. In the same work, he also established a forbidden submatrix characterization for the \(D\)-circular property and presented a linear-time recognition algorithm for it.
	
	In this paper, we introduce a matrix property that extends and refines the \(D\)-circular property, which we call the \emph{\(I\)-circular property}. This property requires that the rows of a matrix, as well as all pairwise intersections of these rows, form circular intervals under some linear ordering of the columns. This refinement serves as a key tool for characterizing semi-transitive split graphs and offers a natural matrix-theoretic framework for identifying forbidden configurations.
	
	The main objective of this work is to establish a \emph{minimal forbidden submatrix characterization} for the \(I\)-circular property, thereby shedding light on the minimal forbidden induced subgraph characterization of semi-transitive split graphs. To achieve this, we develop a unified framework linking the \(I\)-circular property of matrices to the combinatorial properties of split graphs. This approach culminates in the identification of a complete set of minimal forbidden induced subgraphs.
	
	This paper is organized as follows. In Section~2, we review key definitions and preliminaries on sequences, graphs, and matrices, with a particular focus on the circular-ones property and semi-transitive orientations. Section~3 introduces the \(I\)-circular property, explores its relationship with the circular-ones property, and identifies the set of minimal forbidden submatrices for the \(I\)-circular property. Subsections~3.1–3.3 present several matrix constructions and lemmas that support the main results. Section~4 bridges these matrix-theoretic results to the graph-theoretic domain, establishing a forbidden induced subgraph characterization of semi-transitive split graphs. Finally, Section~5 provides a summary and discusses potential directions for future research.
	
	\section{Definitions and Preliminaries}
	For each positive integer $k$, we denote the set $\{1, 2, \ldots, k\}$ by $[k]$. If $k = 0$, we define $[k]$ to be the empty set. We denote the identity function by $id_k$. A set $S$ is said to be properly contained in $T$ if $S \subseteq T$ and $S \neq T$. Whenever we say that $i$ is modulo $k$, or that addition and subtraction involving $i$ are taken modulo $k$, we mean that $i$, or any addition or subtraction involving $i$, is taken modulo $k$ and yields the unique element $j \in [k]$ with the same remainder as $j$ upon division by $k$.
	\subsection{Sequences}
	Let $a = a_1 a_2 \ldots a_k$ be a sequence of length $k$. The shift of $a$ is the sequence $a_2 a_3 \ldots a_k a_1$. The length of any sequence $a$ is denoted by $|a|$. A \emph{binary bracelet}~\cite{MR1857399} is a lexicographically smallest element in an equivalence class of binary sequences under shifts and reversals. For each $k\geq 4$, let $A_k$ be the set of binary bracelets of length $k$. Let $A_3=\{000,111\}$. The elements of $A_3$ are binary bracelets but the two other binary bracelets of length $3$ ($001$ and $011$) do not belong to $A_3$. If $b$ is a sequence and $i \in [k]$, then $b$ is said to occur circularly in $a$ at position $i$ if $|b| \leq k$ and $a_i a_{i+1} \ldots a_{i + |b| - 1} = b$, where the subscripts are taken modulo~$k$. If $a$ is binary (that is, each $a_i \in \{0,1\}$), then the complement of $a$ is defined as the sequence $\overline{a}$ obtained from $a$ by interchanging $0$’s and $1$’s. An index map of $a$ is an injective function $\rho \colon [k'] \to [k]$. If $\rho$ is an index map of $a$, then the sequence $a_{\rho(1)} a_{\rho(2)} \ldots a_{\rho(k')}$ is denoted by $a_{\rho}$. A sequence is quaternary if each of its elements lies in the set $\{0,1,2,3\}$.
	
	\subsection{Graphs}
	Let $G$ be a graph. We denote the vertex set of $G$ by $V(G)$ and the edge set of $G$ by $E(G)$. Let $v \in V(G)$. The neighborhood of $v$ in $G$, denoted by $N_G(v)$, is the set of vertices adjacent to $v$ in $G$. The closed neighborhood of $v$ in $G$, denoted by $N_G[v]$, is the set $N_G(v) \cup \{v\}$. If $X \subseteq V(G)$, the subgraph of $G$ induced by $X$ is the graph whose vertex set is $X$ and whose edge set consists of all edges of $G$ whose endpoints both lie in $X$. If $H$ is a graph, we say that $G$ contains $H$ as an induced subgraph if $H$ is isomorphic to some induced subgraph of $G$. An independent set (respectively, clique) of a graph $G$ is a set of vertices of $G$ that are pairwise nonadjacent (respectively, pairwise adjacent). A graph class is hereditary if it is closed under taking induced subgraphs.  
	
	A graph $G$ is called a split graph if its vertex set $V(G)$ can be partitioned into two sets $I$ and $C$ such that $I$ induces an independent set and $C$ induces a clique. We assume that no vertex from the independent set $I$ is adjacent to all vertices of the clique $C$.  
	
	All graphs in this study are assumed to be simple, that is, finite, undirected, and having no loops or multiple edges. For additional graph-theoretic terminology and notation, we refer the reader to \cite{MR1367739}.
	\subsection{Matrices}
	
	All the matrices considered in this article are binary, that is, every entry of the matrices is either $0$ or $1$. Each row $r$ of a matrix $M$ is defined as the set of columns of $M$ having a $1$ in row $r$. A row $r$ is said to be contained in a row $s$, denoted by $r \subseteq s$, if $s$ has a $1$ in every column where $r$ has a $1$. Similarly, if $r$ and $s$ are rows of $M$, then $r \cap s$ denotes the set of columns having a $1$ in both $r$ and $s$. A row is empty if all its entries are $0$. A row is trivial if all its entries are equal to the same value. Analogous conventions apply to the columns of $M$.
	
	Let $M$ be a $k \times l$ matrix. The rows and columns of $M$ are labeled from $1$ to $k$ and from $1$ to $l$, respectively. The complement of row $i$ is obtained by interchanging $0$’s and $1$’s in that row. The complement of $M$, denoted by $\overline{M}$, is the matrix obtained from $M$ by complementing every row of $M$. If $a$ is a binary sequence of length $k$, we denote by $a \miop M$ the matrix obtained from $M$ by complementing all rows $i \in [k]$ such that $a_i = 1$. A row map of a matrix is an injective function $\rho \colon [k'] \to [k]$ for some positive integer $k'$. A column map of a matrix is an injective function $\sigma \colon [l'] \to [l]$ for some positive integer $l'$. If $\rho \colon [k'] \to [k]$ is a row map and $\sigma \colon [l'] \to [l]$ is a column map, then we denote by $M_{\rho, \sigma}$ the $k' \times l'$ matrix obtained from $M$ such that the $(i,j)$-th entry of $M_{\rho, \sigma}$ is the $(\rho(i), \sigma(j))$-th entry of $M$.
	
	Let $M$ and $M'$ be matrices. $M$ contains $M'$ as a configuration if some submatrix of $M$ equals $M'$ up to permutations of rows and of columns. Notably, a matrix $M$ is said to contain $M'$ as a configuration if and only if there exist a row map $\rho$ and a column map $\sigma$ of $M$ such that $M_{\rho,\sigma} = M'$. We say that $M$ and $M'$ represent the same configuration if $M$ and $M'$ equals up to permutations of rows and of columns. 
	
	If $a$ is a binary sequence whose length equals the number of rows of $M$, then $(a \miop M)_{\rho,\sigma} = a_\rho \miop M_{\rho,\sigma}. $ Moreover, if $\rho$ and $\sigma$ are a row map and a column map of $M$, and $\rho'$ and $\sigma'$ are a row map and a column map of $M_{\rho,\sigma}$, then the compositions $\rho' \circ \rho$ and $\sigma' \circ \sigma$ are a row map and a column map of $M$, and $M_{\rho' \circ \rho,\, \sigma' \circ \sigma} = \left(M_{\rho,\sigma}\right)_{\rho',\sigma'}. $ If $s$ is a positive integer and $n_1, n_2, \ldots, n_s$ are pairwise distinct positive integers, the injective function with domain $[s]$ that maps $i$ to $n_i$ for each $i \in [s]$ is denoted by $\langle n_1, n_2, \ldots, n_s \rangle$. In particular, if $i \in [k]$, then $M_{\langle i \rangle}$ denotes the $1 \times \ell$ matrix whose only row equals row $i$ of $M$. If $M$ is a matrix, then $M^*$ denotes the matrix obtained from $M$ by adding an empty column.
	
	\subsection{Circular-Ones Property}
	\begin{definition}\label{ciri}
		If $\preccurlyeq_X$ is a linear order on some set $X$ and $a,b\in X$, then the \emph{circular interval of $\preccurlyeq_X$ with left endpoint $a$ and right endpoint $b$}, denoted $[a,b]_{\preccurlyeq_X}$, is either $\{x\in X\colon\;a\preccurlyeq_X x\preccurlyeq_X b\}$ if $a\preccurlyeq_X b$, or $\{x\in X\colon\,x\preccurlyeq_X b\text{ or }a\preccurlyeq_X x\}$ if $b\preccurlyeq_X a$. A \emph{circular interval of $\preccurlyeq_X$} is either the empty set or $[a,b]_{\preccurlyeq_X}$ for some $a,b\in X$.
	\end{definition}
	\begin{definition}\label{cir}
		A matrix $M$ has the circular-ones property for rows if there is a linear order $\preccurlyeq_C$ of the columns of $M$ such that each row of $M$ is a circular order of $\preccurlyeq_C$, and $\preccurlyeq_C$ is called a circular-ones order. Analogous definition apply to the columns of $M$. If no mention is made for rows or columns, we mean the corresponding property for the rows.
	\end{definition}
	\begin{figure}[h]
		\ffigbox[\textwidth]{%
			\begin{subfloatrow}
				\hspace{1.25cm}\subfloat[$\MI k$ for each $k\geq 3$]{%
					\hspace{-1.25cm}$\MI k=\left(\begin{array}{ccccc}
						1 & 1 \\
						& 1 & 1 \\
						&   & \ddots & \ddots \\
						&   &        &     1 & 1\\
						1 & 0 & \cdots &     0 & 1
					\end{array}\right)$
				}\qquad
				\hspace{1.25cm}\subfloat[t][$\MII k$ for each $k\geq 4$]{%
					\hspace{-1.25cm}\begin{math}
						\MII k=\left(\begin{array}{cccccc}
							1  & 1  &        &        &   & 0\\
							& 1  & 1      &        &   & 0\\
							&    & \ddots & \ddots &   & \vdots\\
							&    &        &      1 & 1 & 0\\
							1  & 1  & \cdots &      1 & 0 & 1\\
							0  & 1  & \cdots &      1 & 1 & 1
						\end{array}\right)
					\end{math}
			}\end{subfloatrow}
			
			\begin{subfloatrow}
				\hspace{1.25cm}\subfloat[$\MIII k$ for each $k\geq 3$]{%
					\hspace{-1.25cm}\begin{math}
						\MIII k=\left(\begin{array}{cccccc}
							1  & 1  &        &        &   & 0\\
							& 1  & 1      &        &   & 0\\
							&    & \ddots & \ddots &   & \vdots\\
							&    &        &      1 & 1 & 0\\
							0  & 1  & \cdots &      1 & 0 & 1\\
						\end{array}\right)
					\end{math}
				}\qquad
				\hspace{1.25cm}\subfloat[$\MIV$]{%
					\hspace{-1.25cm}	\begin{math}
						\MIV = \left(\begin{array}{cccccc}
							1 & 1 & 0 & 0 & 0 & 0\\
							0 & 0 & 1 & 1 & 0 & 0\\
							0 & 0 & 0 & 0 & 1 & 1\\
							0 & 1 & 0 & 1 & 0 & 1
						\end{array}\right)
					\end{math}
				}
			\end{subfloatrow}
			
			\hspace{1cm}\subfloat[$\MV$]{%
				\hspace{-1cm}\begin{math}
					\MV = \left(\begin{array}{cccccc}
						1 & 1 & 0 & 0 & 0\\
						1 & 1 & 1 & 1 & 0\\
						0 & 0 & 1 & 1 & 0\\
						1 & 0 & 0 & 1 & 1
					\end{array}\right)
				\end{math}
			}
		}{\caption{Tucker matrices, where $k$ denotes the number of rows and omitted entries are $0$'s}\label{fig:TuckerMatrices}}
	\end{figure}
	Tucker~\cite{tucker1972structure} characterized the consecutive-ones property by a minimal set of forbidden submatrices, known as \emph{Tucker matrices}. The Tucker matrices are displayed in Figure~\ref{fig:TuckerMatrices}. In~\cite{mdsafe1}, Mart\'in D. Safe gave an analogous characterization for the circular-ones property. The corresponding set of forbidden submatrices is
	\[ \ForbRow=\{a\miop\MIast k\colon\,k\geq 3\mbox{ and }a\in A_k\}\cup\{\MIV,\overline{\MIV},\MVast,\overline{\MVast}\}, \]
	where $\MIast k$ and $\MVast$ denote $(\MI k)^*$ and $(\MV)^*$, $A_3=\{000,111\}$ and, for each $k\geq 4$, $A_k$ is the set of all binary bracelets of length $k$. A matrix $M$ is a \emph{minimal forbidden submatrix for the circular-ones property} if $M$ is the only submatrix of $M$ not having the circular-ones property.
	\begin{theorem}[\cite{mdsafe1}]\label{circr} A matrix $M$ has the circular-ones property if and only if $M$ contains no matrix in the set $\ForbRow$ as a configuration. Moreover, there is a linear-time algorithm that, given any matrix $M$ not having the circular-ones property, outputs a matrix in $\ForbRow$ contained in $M$ as a configuration. In addition, every matrix in $\ForbRow$ is a minimal forbidden submatrix for the circular-ones property. Hence, for each $M\in\ForbRow$ and each binary sequence $a$ whose length equals the number of rows of $M$, $a\miop M$ represents the same configuration as some matrix in $\ForbRow$. 
	\end{theorem}
	\subsection{Semi-Transitive Orientability of Split Graphs}
	A \textit{semi-transitive} orientation is an acyclic orientation such that, for any directed path $v_0 \rightarrow v_1 \rightarrow \cdots \rightarrow v_m$ with $m \ge 2$, either there is no arc between $v_0$ and $v_m$, or for all $1 \le i < j \le m$, there exists an arc from $v_i$ to $v_j$. An induced subgraph on at least four vertices $\{v_0, v_1, \ldots, v_k\}$ of an oriented graph is called a \textit{shortcut} if it is acyclic, non-transitive, and contains both the directed path $v_0 \rightarrow v_1 \rightarrow \cdots \rightarrow v_k$ and the arc $v_0 \rightarrow v_k$, which is called the \textit{shortcutting edge}. A semi-transitive orientation can thus be equivalently defined as an acyclic, shortcut-free orientation. An undirected graph is said to be semi-transitive if it admits a semi-transitive orientation. The class of semi-transitive graphs is hereditary.
	
	Semi-transitive orientability of split graphs are recently studied in the papers \cite{kitaev2008word, iamthong2021semi, iamthong2022word, kitaev2024semi, chen2022representing}. Some split graphs are semi-transitive and some are not. Given a split graph $G$ with $C = \{u_1, u_2, \ldots, u_n\}$ and $I = \{v_1, v_2, \ldots, v_m\}$, consider a $(0,1)$-matrix $A(G)$ of size $m \times n$, where $a_{ij} = 1$ if and only if $v_i$ is adjacent to $u_j$. In \cite{kitaev2024semi}, Kitaev and Pyatkin have characterized the semi-transitive split graphs using following theorem.
	\begin{theorem}[\cite{kitaev2024semi}, Theorem 3]
		A split graph $G$ is semi-transitive if and only if
		\begin{itemize}
			\item $A(G)$ has circular-ones property for rows.
			\item If a row of $A(G)$ has the form $1^a 0^b 1^c$
			where $a + b + c = k$ and $a, b, c \geq 1$, then no other row
			may contain ones in all positions from $a$ to $a + b + 1$.
		\end{itemize}
	\end{theorem}
	\section{The \(I\)-Circular Property}
	
	In the previous section, we reviewed the necessary preliminaries on graph theory, matrix properties, and semi-transitive orientations. Building on these foundations, we now introduce the \(I\)-circular property, which plays a central role in our characterization of minimal forbidden induced subgraphs for semi-transitive split graphs. The \(I\)-circular property is a matrix property closely related to the well-known \(D\)-circular property introduced by Safe~\cite{mdsafe2}. In this section, we formally define the \(I\)-circular property and establish a minimal forbidden submatrix characterization for it.
	
	 \begin{definition}
		A matrix $M$ has the $I$-circular property if there is a linear order
		$\preccurlyeq_C$ of its columns such that each row of $M$ is a circular interval of $\preccurlyeq_C$ and the set intersection $s \cap r$ is also a circular interval of $\preccurlyeq_C$ for each pair of rows $r$ and $s$ of $M$.
		If so, $\preccurlyeq_C$ is a $I$-circular order of $M$.
	\end{definition}
	\begin{definition}
		We denote by $\Lambda$ an operator that assigns to each matrix $M$ some matrix $\Lambda(M)$ that arises from $M$ by adding rows at the bottom as follows: for each	pair of nontrivial rows $r$ and $s$ such that complement of $r$ is properly contained in $s$, add a row equal to the intersection of rows $r$ and $s$. We do not specify the order in which these rows are added
		because it is immaterial for our purposes.
	\end{definition}
	\begin{lemma}\label{IM} A matrix $M$ has the $I$-circular property if and only if $\Lambda(M)$ has the circular-ones property.\end{lemma}
	\begin{proof} By definition, each $I$-circular order of $M$ is a circular-ones order of $\Lambda(M)$. Hence, if $M$ has the $I$-circular property, then $\Lambda(M)$ has the circular-ones property. For the converse, suppose that $\Lambda(M)$ has the circular-ones property. Thus, there is some circular-ones order $\preccurlyeq_C$ of $\Lambda(M)$. Let $r$ and $s$ be two rows of $M$. As $r$ and $s$ are rows of $\Lambda(M)$, then $r$ and $s$ are circular intervals of $\preccurlyeq_C$. Hence, $s\cap r$ is also a circular interval of $\preccurlyeq_C$ unless, perhaps, when complement of $r$ is properly contained in $s$.
		
		 If $r$ or $s$ is trivial, then $s\cap r$ is trivial, $s$, or $r$, all of which are circular intervals of $\preccurlyeq_C$. Thus, we assume, without loss of generality, that $r$ and $s$ are nontrivial. Therefore, if the complement of $r$ is properly contained in $s$, then $s\cap r$ is a circular interval of $\preccurlyeq_C$ because $s \cap r$ is a row of $\Lambda(M)$. We conclude that in all cases $s\cap r$ is a circular interval of $\preccurlyeq_C$. Thus, by definition, $\preccurlyeq_C$ is a $I$-circular order of $M$. This proves that, whenever $\Lambda(M)$ has the circular-ones property, $M$ has the $I$-circular property. The proof of the lemma is complete.\end{proof}
		 
		 \begin{figure}[h]
		 			\[ \MVI = \begin{pmatrix}
		 				1 &  1 & 0 & 1\\
		 				0  & 1 & 1 & 1\\
		 				1 & 0 & 1 & 1\\
		 			\end{pmatrix}\]
		 			\caption{$\MVI$}
		 \end{figure}
	
	The set of minimal forbidden submatrices in the characterisation for the \(I\)-circular property is the set $\IForbRow$ defined as follows.
	\begin{definition}
		We denote \[ \IForbRow=\{\MIast k, \MIII{k}, \MII{k+1}\colon\,k\geq 3\}\cup\{0101~\miop \MIast{4}, 0100~\miop\MII{4}, \MIV,\MV,\MVI\}, \]
	\end{definition}
	 We now state the following lemmas, which help us understand the relationship between the set \(\IForbRow\) and the set \(\ForbRow\).
	\begin{lemma}\label{m2}
		Let $a = a_1 a_2 \ldots a_{k-2}~0~0$ be a binary sequence of length $k$ for some $k \ge 4$. If $a \not = 0100$, then $a \miop \MII{k}$ contains some matrix in $\IForbRow$ as a configuration.
	\end{lemma}
	\begin{proof}
	Let $a = a_1 a_2 \ldots a_{k-2}\,0\,0$ be a binary sequence of length $k$ for some $k \geq 4$. If $a_i = 0$ for all $1 \leq i \leq k-2$, then $a \miop \MII{k} = \MII{k}$. Now suppose that $a_i = 1$ for some $2 \leq i \leq k-3$. It can be verified that $(a \miop \MII{k})_{<i,\, k-1,\, k>,\, <k-1,\, 1,\, i,\, k>} = \MVI.$
	
	Next, suppose that $a_i = 0$ for all $2 \leq i \leq k-3$, and that either $a_1 = 1$ or $a_{k-2} = 1$. 
	
	- If $a_1 = 1$ and $a_{k-2} = 0$, then
	\[
	(a \miop \MII{k})_{<2,\, 3,\, \ldots,\, k-1,\, 1>,\, <2,\, 3,\, \ldots,\, k>} = \MII{k-1}.
	\]
	
	- If $a_1 = 0$ and $a_{k-2} = 1$, then
	\[
	(a \miop \MII{k})_{<1,\, 2,\, 3,\, \ldots,\, k-2,\, k>,\, <1,\, 2,\, 3,\, \ldots,\, k-2,\, k>} = \MII{k-1}.
	\]
	
	- If $a_1 = a_{k-2} = 1$ and $k \geq 6$, then it can be verified that
	\[
	(a \miop \MII{k})_{<2,\, 3,\, \ldots,\, k-2,\, 1>,\, <2,\, 3,\, \ldots,\, k-2,\, k>} = \MII{k-2}.
	\]
	
	- If $a_1 = a_{k-2} = 1$ and $k = 5$, then it can be verified that
	\[
	(10100 \miop \MII{5})_{<2,\, 1,\, 4,\, 3>,\, <3,\, 2,\, 1,\, 4>} = 0100 \miop \MII{4}.
	\]
	Now, the only remaining cases are when $a = 1100$ and $a = 1000$:
	
	- If $a = 1100$, then 
	\[
	(1100 \miop \MII{4})_{<1,\, 2,\, 4,\, 3>,\, <3,\, 4,\, 1,\, 2>} = \MII{4}.
	\]
	
	- If $a = 1000$, then 
	\[
	(1000 \miop \MII{4})_{<2,\, 1,\, 4,\, 3>,\, <3,\, 2,\, 1,\, 4>} = 0100 \miop \MII{4}.
	\]
	
	This completes the proof of the lemma.
	
	\end{proof}
	\begin{lemma}\label{L1}
		If $F \in \IForbRow$, then $\Lambda(F)$ contains some matrix in $\ForbRow$ as a configuration. Moreover, $F$ is a minimal forbidden submatrix for $I$-circular property.
	\end{lemma}
	\begin{proof}
	If $F = \MIast{k}$ for some $k \geq 3$, or if $F = \MIII{k}$ for some $k \geq 3$, or if $F = \MIV$, then the lemma holds immediately, as $\Lambda(F) = F$ and $F \in \ForbRow$. Moreover, since, by \textit{Theorem~\ref{circr}}, $\Lambda(F)$ is a minimal forbidden submatrix for the circular-ones property, $F$ is a minimal forbidden submatrix for the $I$-circular property.
	
	If $F = 0101 \miop \MIast{4}$, then $\Lambda(F)_{<1, 2, 3, 4>, id_5} = F$ and $F \in \ForbRow$. Furthermore,\\ $\Lambda(F)_{<1, 2, 3, 4>, id_5}$ is the only submatrix of $\Lambda(F)$ that lacks the circular-ones property. By \textit{Theorem~\ref{circr}}, $\Lambda(F)$ is minimal for the circular-ones property, and thus $F$ is minimal for the $I$-circular property.
	
	Now consider the remaining cases. For each, note that every row and column of $F$ is required to obtain a submatrix of $\Lambda(F)$ that has the same configuration as some matrix in $\ForbRow$ and is also minimal for the circular-ones property by \textit{Theorem~\ref{circr}}. Thus, in each case, $F$ itself is minimal for the $I$-circular property.
	\begin{itemize}
		\item If $F = \MII{k}$ for some $k \geq 4$, let row $k+1$ of $\Lambda(F)$ be obtained by intersecting rows $k-1$ and $k$ of $F$. With $\rho = <1, 2, \ldots, k-2, k+1>$, we have $\Lambda(F)_{\rho, id_k} = 00\ldots01 \miop \MIast{k-1}$.
		\item If $F = \MV$, $\Lambda(F)$ adds a sixth row from rows $2$ and $4$ of $F$, and $\Lambda(F)_{<1, 4, 3, 6>, id_5} = 0100 \miop \MIast{4}$.
		\item If $F = 0100 \miop \MII{4}$, $\Lambda(F)$ adds a fifth row from rows $3$ and $4$, and $\Lambda(F)_{<1, 2, 6>, id_4} = \MIast{3}$.
		\item If $F = \MVI$, $\Lambda(F)$ adds rows $4$, $5$, and $6$ from intersections of rows $1$ and $2$, $1$ and $3$, and $2$ and $3$, respectively, yielding $\Lambda(F)_{<4, 5, 6>, id_4} = 111 \miop \MIast{3}$.
	\end{itemize}
	This completes the proof that $F \in \IForbRow$ is a minimal forbidden submatrix for the $I$-circular property in all cases.

	\end{proof}
	\begin{corollary}\label{NFIM}
		None of the matrices in $\IForbRow$ has the $I$-circular property
	\end{corollary}
	\begin{proof}
	Let $F$ be any matrix in $\IForbRow$. By \textit{Lemma}~\ref{L1}, $\Lambda(F)$ contains some matrix in $\ForbRow$ as a configuration. Thus, by \textit{Theorem}~\ref{circr}, $\Lambda(F)$ does not have the circular-ones property. Hence, \textit{Lemma}~\ref{IM} implies that $F$ does not have the $I$-circular property. This completes the proof of the corollary.
	\end{proof}
	\begin{lemma}\label{fc}
		If $F \in \ForbRow$, then $F$ contains some matrix in $\IForbRow$ as a configuration.
	\end{lemma}
	\begin{proof}
		Let $F$ be some matrix in $\ForbRow$. If $F$ is $\MIV$, then $F \in \IForbRow$. If $F$ is $\MVast$, then $F_{id_5, <1,2,3,4,5>} = \MV$. If $F$ is $\overline{\MIV}$, then $F_{<2, 3, 4>, <1, 2, 3, 5>} = \MVI$. If $F$ is $\overline{\MVast}$, then $F_{<1, 3, 4>, <2, 3, 5, 6>} = \MVI$. Therefore, it only remains to consider the case where $F = a \miop\MIast{k}$ for some binary sequence $a = a_1 a_2 \ldots a_k$ such that $a \in A_k$. If $a$ is empty, then $F = \MIast{k}$. Therefore, assume that $1$ occurs at least once in $a$.
		
		Suppose that $k \geq 5$ and $1x11$ occurs circularly in $a$ at position $i$ for some $i \in [k]$ and for some $x \in \{0, 1\}$. If we let $\rho = <i, i+2, i+4>$ and $\sigma = <i, i+2, i+4, k+1>$ (where the sums involving $i$ are modulo $k$), then $a_{\rho}  = 111$,
		\[
		F_{\rho, \sigma} = (a \miop \MIast{k})_{\rho, \sigma} = 111 \miop \MIast{k}_{\rho, \sigma},
		\]
		and, consequently, $F_{\rho, \sigma}$ has the same configuration as $\MVI$.
		
		Now suppose that $k \geq 5$ and $11x1$ occurs circularly in $a$ at position $i$ for some $i \in [k]$ and for some $x \in \{0, 1\}$. If we let $\rho = <i, i+1, i+4>$ and $\sigma = <i, i+2, i+4, k+1>$ (where the sums involving $i$ are modulo $k$), then $a_{\rho}  = 111$,
		\[
		F_{\rho, \sigma} = (a \miop \MIast{k})_{\rho, \sigma} = 111 \miop \MIast{k}_{\rho, \sigma},
		\]
		and, consequently, $F_{\rho, \sigma}$ has the same configuration as $\MVI$.
		
		Now suppose that $k \geq 4$ and $1 0^m 1$ occurs circularly in $a$ at position $i$ for some $i \in [k]$ and some $m \geq 2$. If we let $\rho = <i+1, i+2, \ldots, i+m+1, i>$ and $\sigma = <i+1, i+2, \ldots, i+m+1, k+1>$ (where the sums involving $i$ are modulo $k$), then $F_{\rho, \sigma}$ has the same configuration as $\MII{k}$.
		
		Now suppose that $k \geq 5$ and $1 0 1 0 1$ occurs circularly in $a$ at position $i$ for some $i \in [k]$. If we let $\rho = <i, i+2, i+4>$ and $\sigma = <i+1, i+3, i+ 4, k+1>$ (where the sums involving $i$ are modulo $k$), then $F_{\rho, \sigma}$ has the same configuration as $\MVI$.
		
		Now suppose that $k \geq 4$ and $1$ occurs twice in $a$. First, assume that $11$ occurs circularly in $a$ at position $i$ for some $i \in [k]$. If we let $\rho = <i+2, \ldots, k, i, i+1>$ and $\sigma = <i+2, \ldots, k, 1, 2, \ldots, i, k+1>$ (where the sums involving $i$ are modulo $k$), then $F_{\rho, \sigma}$ has the same configuration as $\MII{k}$.	Now assume that $101$ occurs circularly in $a$ at position $i$ for some $i \in [k]$. If we let $\rho = <i+3, i+4, \ldots, k, 1, 2, \ldots, i, i+2>$ and $\sigma = <i+3, i+4, \ldots, k, 1, 2, \ldots, i, k+1>$ (where the sums involving $i$ are modulo $k$), then $F_{\rho, \sigma}$ has the same configuration as $\MII{k-1}$.
		
		Now suppose that $k \geq 4$ and $1$ occurs once in $a$ at position $i$ for some $i \in [k]$. If we let $\rho = <i+1, i+2, \ldots, k, 1, 2, \ldots, i>$ and $\sigma = <i+1, i+2, \ldots, k, 1, 2, \ldots, i, k+1>$ (where the sums involving $i$ are modulo $k$), then $F_{\rho, \sigma}$ has the same configuration as $\MIII{k}$.
		
		Now suppose that $k \ge 5$ and $1$ occurs thrice in $a$. Then, the only case left is when $111$ occurs circularly in $a$ at position $i$ for some $i \in [k]$. If we let $\rho = <i+3, i+4, \ldots, k, 1, 2, \ldots, i>$ and $\sigma = <i+3, i+4, \ldots, k, 1, 2, \ldots, i, k+1>$ (where the sums involving $i$ are modulo $k$), then $F_{\rho, \sigma}$ has the same configuration as $\MII{k-1}$. 
		
		Now suppose that $k = 4$. Since $a$ is a binary bracelet of length $4$, the only cases left are when $a = 0101$ ,$a = 0111$ and $a = 1111$. If $a = 0101$, then $F \in \IForbRow$. If $a = 0111$, then $F_{<1, 2, 3, 4>, <1, 2, 3, 5>} = 0100 \miop \MII{4}$.  If $a = 1111$, then $F_{<1, 2, 3, 4>, <3, 5, 1, 2>} = 0100 \miop \MII{4}$. 
		
		Finally, suppose that $k=3$. By \textit{Theorem }\ref{circr}, we may assume that $a = 111$. Then, it can be verified that $111 \miop\MIast{3} = \MIII{3}$. This completes the proof of the lemma.   
	\end{proof}
	
	We state the following lemma which will be used in the proof of the main result.
	\begin{lemma}\cite{mdsafe1}\label{MVast} If $a$ is any binary sequence of length $4$, then $a\miop\MVast$ represents the same configuration as one of the matrices $\MIV$, $\overline{\MIV}$, $\MVast$, and $\overline{\MVast}$. Conversely, each of the matrices $\MIV$, $\overline{\MIV}$, $\MVast$, and $\overline{\MVast}$ represents the same configuration as $a\miop\MVast$ for some binary sequence $a$ of length $4$. Moreover, the four matrices $\MIV$, $\overline{\MIV}$, $\MVast$, and $\overline{\MVast}$ represent pairwise different configurations.\end{lemma}
	\subsection{Matrices $Q$ and $R$}
	
	To each quaternary sequence $b$ of length at least $3$, we associate a matrix denoted by $R(b)$. As a preliminary step toward the proof of the Theorem, we establish the following Lemma, which states that for almost all such quaternary sequences $b$, the matrix $R(b)$ contains some matrix in $\IForbRow$ as a configuration. We now introduce the necessary definitions to formalize this.
	
	For each $k\geq 3$ and each $i\in[k]$, we define the following matrices, where in all the cases $i+1$ should be understood modulo $k$:
	\begin{itemize}
		\item $Q_0(i,k)$ is the $1\times(k+1)$ matrix whose only row has $1$'s at columns $i$ and $i+1$ and $0$'s at the remaining ones;
		
		\item $Q_1(i,k)$ is the complement of $Q_0(i,k)$;
		
		\item $Q_2(i,k)$ is the $2\times(k+1)$ matrix whose first row has a $0$ at column $k+1$ and $1$'s at the remaining columns and whose second row has $1$'s at columns $i$, $i+1$, and $k+1$ and $0$'s at the remaining columns;
		
		\item $Q_3(i,k)$ is the $2\times(k+1)$ matrix whose first row has a $0$ at column $i$ and $1$'s at the remaining columns and whose second row has $0$'s at columns $i+1$ and $1$'s at the remaining columns.
	\end{itemize}
	
	Given a quaternary sequence $b = b_1 b_2 \ldots b_k$ of length $k$ for some $k \geq 3$, we define $R(b)$ to be the matrix with $k+1$ columns whose rows consist of the rows of $Q_{b_1}(1, k)$, followed by those of $Q_{b_2}(2, k)$, then $Q_{b_3}(3, k)$, and so on, up to the rows of $Q_{b_k}(k, k)$. For example,
	\[ R(013102)=
	\begin{pmatrix}
		1 & 1 & 0 & 0 & 0 & 0 & 0\\
		1 & 0 & 0 & 1 & 1 & 1 & 1\\
		1 & 1 & 0 & 1 & 1 & 1 & 1\\
		1 & 1 & 1 & 0 & 1 & 1 & 1\\
		1 & 1 & 1 & 0 & 0 & 1 & 1\\
		1 & 1 & 1 & 1 & 1 & 1 & 0\\
		1 & 0 & 0 & 0 & 0 & 1 & 1
	\end{pmatrix}. \]
	\begin{lemma}\label{rb}
		Let $b = b_1 b_2 \ldots b_k$ be a quaternary sequence of length $k$ such that $k \geq 3$. If $k = 3$, suppose additionally that either $b_1, b_2, b_3 \in \{1, 3\}$ or $b_1, b_2, b_3 \in \{0, 2\}$. Then, $R(b)$ contains some matrix in $\IForbRow$ as a configuration.
	\end{lemma}
	\begin{proof}
		First, we show that for any quaternary sequence $c$ obtained from $b$ by a sequence of shift operations, the matrix $R(c)$ represents the same configuration as $R(b)$. To prove this, it suffices to show that $R(c)$ represents the same configuration as $R(b)$ when $c$ is a single shift of $b$ (since the general case then follows by induction). Let $m$ denote the number of rows of $R(b)$, and suppose that $c$ is a shift of $b$. Define the row permutation $\rho$ as $\langle 2, 3, \ldots, m, 1 \rangle$ if $b_1 \in \{0, 1\}$, and as $\langle 3, 4, \ldots, m, 1, 2 \rangle$ if $b_1 \notin \{0, 1\}$. Also, define the column permutation $\sigma$ as $\langle 2, 3, \ldots, k, 1, k+1 \rangle$. Then, we have $R(c) = R(b)_{\rho, \sigma}$.

		Let $b = b_1 b_2 \ldots b_k$ be a quaternary sequence of length $k$, where $k \geq 3$. For each $i \in [k]$, let $f_i$ denote the row index of $R(b)$ corresponding to the first row of $Q_{b_i}(i, k)$. If $b_i \notin \{0, 1\}$, then let $s_i = f_i + 1$ denote the row index of $R(b)$ corresponding to the second row of $Q_{b_i}(i, k)$.
		
		\textbf{Case 1:} Let $b$ be a binary sequence. By construction, we have $R(b) = b \miop \MIast{k}$ for some $k \geq 3$. By applying a sequence of shift operations (if necessary), we may assume, without loss of generality, that $b$ is a bracelet. Therefore, $R(b) \in \IForbRow$.
		
		\textbf{Case 2: }Suppose that $3$ occurs more that once in $b$. Let $b_i = b_j = 3$, where $i, j \in [k]$. Without loss of generality, we may assume that $i < j$. It can be verified that \\ $R(b)_{<f_i, s_i, f_j>, <i, i+1, j+1, k+1>} = \MVI$, when $i +1 \not = j$, and $R(b)_{<f_i, s_i, s_j>, <i, i+1, j+1, k+1>} = \MVI$, when $i +1 = j$, where all additions involving $i$ and $j$ are taken modulo $k$.
		
		\textbf{Case 3: }Suppose that both $2$ and $3$ occur in $b$. Let $b_i = 2$ and $b_j = 3$. It can be verified that $R(b)_{<f_i, s_i, f_j>, <i, i+1, i+2, k+1>} = \MVI$, where all additions involving $i$ and $j$ are taken modulo $k$.
		
		\textbf{Case 4: }Suppose $3$ occurs in $b$. By Case~2, it is evident that $3$ occurs exactly once in $b$. By applying a sequence of shift operations (if necessary), we may assume, without loss of generality, that $b_k = 3$, and that $b' = b_1 b_2 \ldots b_{k-1}$ is a binary string. Therefore, $R(b)$ has the same configuration as $a \miop \MII{k+1}$, where $a$ is obtained by appending $00$ to $b'$.  
		
		If $a \not = 0100$, then by \textit{Lemma~\ref{m2}}, $R(b)$ contains some matrix in $\IForbRow$ as a configuration. If $a = 0100$, then $R(b) = 0100 \miop \MII{4}$.
		
		\textbf{Case 5: }Suppose that $2$ occurs in $b$. By applying a sequence of shift operations (if necessary), we may assume, without loss of generality, that $b_1 = 2$. If $b_2 = 1$, then it can be verified that $R(b)_{<1, 2, 3>, <1, 2, k, k+1>} = \MVI.$ If $b_2 = 2$, then it can be verified that $R(b)_{<2, 3, 4>, <1, 2, 3, k+1>} = \MVI.$ If $b_k = 1$, then 
		$R(b)_{<1, 2, k>, <3, k+1, 1, 2>} = \MVI.$ If $b_k = 2$, then $R(b)_{<1, 2, k+1>, <k, 2, k+1, 1>} = \MVI.$
		
		Now suppose that $b_2 = b_k = 0$. If $k \geq 4$, then it can be easily verified that $R(b)_{<1,\, 2,\, 3,\, k>,\, <1,\, k,\, 3,\, 2,\, k+1>} = \MV.$ If $k = 3$, then by the assumption of the lemma, $b_1, b_2, b_3 \in \{0, 2\}$. By applying a sequence of shift operations (if necessary), we may assume, without loss of generality, that $b = 200$. Then, $R(b)_{<3,\, 4,\, 1,\, 2>,\, <3,\, 2,\, 4,\, 1>} = 0100 \miop \MII{4}.$ This completes the proof of the lemma.

	\end{proof}
	\subsection{Matrices $U$ and $W$}\label{ssec:UW}
	
	To each quaternary sequence $b$ of length $4$, we associate a matrix denoted by $W(b)$. We establish the following Lemma, which states that for certain quaternary sequences $b$, the matrix $W(b)$ contains some matrix in $\IForbRow$ as a configuration. We now introduce the necessary definitions to formalize this.
	
	We first define, for each $i\in[4]$, the following matrices:
	\begin{itemize}
		\item $U_0(i)$ whose only rows coincides with row $i$ of $\MVast$;
		\item $U_1(i)$ is the complement of $U_0(i)$.
	\end{itemize}
	For each $i\in[3]$, we define the following matrix:
	\begin{itemize}
		\item $U_2(i)$ is the $2 \times 6$ matrix whose first row has $0$'s in columns $5 - 2i$ and $6 - 2i$, and $1$'s in the remaining columns, and whose second row has $0$'s in columns $7 - 2i$ and $8 - 2i$, and $1$'s in the remaining columns, where all subtractions involving $i$ are taken modulo $6$.
	\end{itemize}
	For $i = 2$, we define the following matrix:
	\begin{itemize}
		\item $U_3(i)$ is the $2\times 6$ matrix whose first row has $0$ in column 5 and $1$'s in the remaining columns, and whose second row has $0$ in column $6$ and $1$'s in the remaining columns.
	\end{itemize}
		
	For each quaternary sequence $b=b_1 b_2 b_3 b_4$ of length $4$ such that $b_4 = 0$, we define $W(b)$ as the matrix having six columns and whose rows are those of $U_{b_1}(1)$, followed by those of $U_{b_2}(2)$, followed by those of $U_{b_3}(3)$, followed by those of $U_{0}(4)$. For instance,
	\[ W(2310) =
	\begin{pmatrix}
		0 & 0 & 1 & 1 & 1 & 1\\
		1 & 1 & 0 & 0 & 1 & 1\\
		1 & 1 & 1 & 1 & 0 & 1\\
		1 & 1 & 1 & 1 & 1 & 0\\
		1 & 1 & 0 & 0 & 1 & 1\\
		1 & 0 & 0 & 1 & 1 & 0
	\end{pmatrix}. \]
	
	\newcommand\LemW{If $b=b_1 b_2 b_3 b_4$ is a quaternary sequence such that $b_4=0$ and $b_1, b_3 \not = 3$, then $W(b)$ contains some matrix in $\IForbRow$ as a configuration.}
	\begin{lemma}\label{W}\LemW\end{lemma}
	\begin{proof}
		Let $b = b_1 b_2 b_3 b_4$ be a quaternary sequence of length $4$. For each $i \in [4]$, let $f_i$ denote the row index of $W(b)$ corresponding to the first row of $U_{b_i}(i)$. If $b_i \notin \{0, 1\}$, then let $s_i = f_i + 1$ denote the row index of $W(b)$ corresponding to the second row of $U_{b_i}(i)$.
		
		\textbf{Case 1:} Let $b$ be a binary sequence. By construction, we have $W(b) = b \miop \MVast$ which, by \textit{Theorem }\ref{circr}, represents the same configuration as some matrix in $\ForbRow$. By \textit{Lemma }\ref{fc}, $W(b)$ contains some matrix in $\IForbRow$ as a configuration.
		
		\textbf{Case 2:} Suppose that $2$ occurs in $b$, and let $b_i = 2$ for some $i \in [k]$. By construction, $b_4 = 0$. If we define $\sigma_1 = <1, 2, 4, 5>$, $\sigma_2 = <1, 4, 5, 6>$, and $\sigma_3 = <1, 3, 4, 5>$, then it can be verified that $W(b)_{<f_i, s_i, f_4>, \sigma_i}$ has the same configuration as $\MVI$.
		
		\textbf{Case 3:} Suppose that $3$ occurs in $b$. By construction, $b_4 = 0$ and $b_2 = 3$.  If $b_1 = b_3 = 0$, then it can be verified that 
		$W(b)_{<1, 2, 4, 5>, <1, 2, 3, 4, 5>} = \MV$. If $b_1 = 1$, then it can be verified that $W(b)_{<1, 2, 5>, <1, 3, 4, 5>}$ has the same configuration as $\MVI$. If $b_3 = 1$, then it can be verified that $W(b)_{<2, 3, 4>, <2, 3, 5, 6>}$ has the same configuration as $\MVI$. This completes the proof of the lemma.
	\end{proof}
	\subsection{Matrices $H$ and $G$}\label{ssec:HG}
	
	For each binary sequence $\alpha = \alpha_1\alpha_2\alpha_3\alpha_4$ of length $4$ and each $i \in [3]$, we define $H_i(\alpha)$ to be the $6 \times 6$ matrix obtained from $\MVast$ by appending two additional rows (the fifth and sixth rows), both having $1$'s in columns $3 - 2i$ and $4 - 2i$. In the fifth row, the entries in columns $5 - 2i$, $6 - 2i$, $1 - 2i$, and $2 - 2i$ (with all subtractions taken modulo $6$) are given by $\alpha_1$, $\alpha_2$, $\alpha_3$, and $\alpha_4$, respectively. In the sixth row, the corresponding entries are $\overline{\alpha_1}$, $\overline{\alpha_2}$, $\overline{\alpha_3}$, and $\overline{\alpha_4}$, respectively.
	 For instance,
	\[ H_1(\alpha_1\alpha_2\alpha_3\alpha_4)=
	\begin{pmatrix}
		1 & 1 & 0 & 0 & 0 & 0\\
		1 & 1 & 1 & 1 & 0 & 0\\
		0 & 0 & 1 & 1 & 0 & 0\\
		1 & 0 & 0 & 1 & 1 & 0\\
		1 & 1 & \alpha_1 & \alpha_3 & \alpha_3 & \alpha_4 \\
		1 & 1 & \overline{\alpha_1} & \overline{\alpha_2} &\overline{ \alpha_3} & \overline{\alpha_4}
	\end{pmatrix}. \]
	
	For each binary sequence $\gamma=\gamma_1\gamma_2\gamma_3$ of length $3$, we define
	\[ G(\gamma)=
	\begin{pmatrix}
		1 & 1 & 0 & 0 & 0 & 0\\
		1 & 1 & 1 & 1 & 0 & 0\\
		0 & 0 & 1 & 1 & 0 & 0\\
		1 & 0 & 0 & 1 & 1 & 0\\
		1 & \gamma_1 &  \gamma_2 & 1 &  1 & \gamma_3\\
		1 & \overline{\gamma_1} & \overline{\gamma_2} & 1 & 1 & \overline{\gamma_3}
	\end{pmatrix}. \]
	
	\begin{lemma}\label{X}
	Let $\alpha$ be a binary sequence of length $4$ and let $i\in[3]$. If $\alpha\notin\{0000, 0011, 1100, 1111\}$, then $H_i(\alpha)$ contains as a configuration some matrix in $\ForbRow$ having fewer than $6$ columns.
	\end{lemma}
	\begin{proof}
		Let $\alpha = \alpha_1 \alpha_2 \alpha_3 \alpha_4$ be a binary sequence of length $4$. Since $H_i(\overline{\alpha})_{<1, 2, 3, 4, 6, 5>, id_6} = H_i(\alpha)$, the lemma holds for $\alpha$ if and only if it holds for $\overline{\alpha}$. In particular, by complementing $\alpha$ (if necessary), we may assume that $\alpha_1 = 0$. We now consider the following cases.
		
		Case 1: $\alpha = 0001$. In this case,  
		$H_1(\alpha)_{<2, 5, 4>, <6, 5, 3, 1>} = H_2(\alpha)_{<3, 5, 4>, <6, 1, 3, 4>} = \overline{\MIast{3}},$ and $H_3(\alpha)_{<1, 5, 4>, <1, 2, 4, 6>} = \MIast{3}$.

		Case 2: $\alpha = 0010$. In this case, 
		$H_1(\alpha)_{<2, 4, 5>, <2, 4, 5, 6>} = H_2(\alpha)_{<4, 3, 5>, <5, 4, 3, 2>} = \MIast{3}$, and $H_3(\alpha)_{<1, 4, 5>, <5, 3, 2, 1>} = \overline{\MIast{3}}$.
		
		Case 3: $\alpha = 0100$. In this case, $H_1(\alpha)_{<3, 4, 5>, <5, 2, 3, 4>} = H_3(\alpha)_{<2, 5, 4>, <6, 5, 2, 4>} = \overline{\MIast{3}}$, and $H_2(\alpha)_{<1, 5, 4>, <1, 2, 5, 3>} = \MIast{3}$.
		
		Case 4: $\alpha = 0101$. In this case, $H_1(\alpha)_{<2, 5, 4>, <6, 5, 3, 1>} = H_2(\alpha)_{<1, 4, 6>, <4, 3, 2, 1>} = \overline{\MIast{3}}$, and $H_3(\alpha)_{<1, 4, 2, 5>, <1, 2, 6, 5, 4>} = 01111 \miop \MIast{4}$.
		
		Case 5: $\alpha = 0110$. In this case, $H_1(\alpha)_{<3, 6, 5, 2>, <3, 4, 5, 6, 2>} = 0111 \miop \MIast{4}$,\\  $H_2(\alpha)_{<1, 6, 3, 4>, <1, 2, 3, 4, 6>} = 0100 \miop \MIast{4}$, and $H_3(\alpha)_{<1, 6, 4>, <1, 2, 4, 6>} = \MIast{3}$.
		
		Case 6: $\alpha = 0111$. In this case, $H_1(\alpha)_{<6, 3, 4>, <1, 3, 4, 6>} =H_3(\alpha)_{<2, 6, 4>, <1, 3, 5, 6>} = \MIast{3}$, and $H_2(\alpha)_{<1, 6, 4>, <6, 4, 2, 1>} = \overline{\MIast{3}}$. This completes the poof of the lemma.
	\end{proof}
	\begin{lemma}\label{G}
	Let $\gamma$ be a binary sequence of length $3$. If $\gamma$ is nonconstant, then $G(\gamma)$ contains as a configuration some matrix in $\ForbRow$ having fewer than $6$ columns.
	\end{lemma}
	\begin{proof}
		Let $\gamma = \gamma_1 \gamma_2 \gamma_3$ be a binary sequence of length $3$. Since $G(\overline{\gamma})_{<1, 2, 3, 4, 6, 5>, id_6} = G(\gamma)$, the lemma holds for $\gamma$ if and only if it holds for $\overline{\gamma}$. In particular, by complementing $\gamma$ (if necessary), we may assume that $\gamma_1 = 0$. We now consider the following cases.
		
		Case 1: $\gamma = 001$. In this case, $G(\gamma)_{<1, 4, 3, 5>, <2, 1, 4, 3, 6>} = 0001 \miop \MIast{4}$.
		
		Case 2: $\gamma = 010$. In this case, $G(\gamma)_{<1, 5, 2, 4>, <1, 2, 6, 5, 3>} = 0110 \miop \MIast{4}$.
		
		Case 3: $\gamma = 011$. In this case, $G(\gamma)_{<3, 4, 2, 6>, <3, 4, 5, 6, 2>} = 0011 \miop \MIast{4}$. This completes the proof of the lemma.
	\end{proof}
	\begin{theorem}\label{icp}
		A matrix $M$ has $I$-circular property if and only if $M$ contains no matrix in $\IForbRow$ as a configuration.
	\end{theorem}
	\begin{proof}
		Suppose that $M$ has the $I$-circular property. Assume, for the sake of contradiction, that $M$ contains some matrix $F \in \IForbRow$ as a configuration. By \textit{Corollary}~\ref{NFIM}, the matrix $F$ does not possess the $I$-circular property. This, in turn, implies that $M$ itself cannot have the $I$-circular property, leading to a contradiction.
		
		Suppose that $M$ contains no matrix in $\IForbRow$ as a configuration. Assume, for the sake of contradiction, that $M$ does not have the $I$-circular property. Let $k_M$ denote the number of rows of the matrix $M$. By \textit{Lemma}~\ref{IM}, the matrix $\Lambda(M)$ does not have the circular-ones property. Then, by \textit{Theorem}~\ref{circr}, $\Lambda(M)$ contains some matrix in $\ForbRow$ as a configuration. Let $F \in \ForbRow$ be a matrix with the minimum possible number of columns such that $F$ is contained in $\Lambda(M)$ as a configuration. Let $\rho$ and $\sigma$ be a row map and a column map, respectively, such that $\Lambda(M)_{\rho,\sigma} = F$.
		
		\textbf{Case 1: }Firstly, suppose that $F = a \miop \MIast{k}$ for some $k \ge 3$ and some $a \in A_k$. Assume that $a_i = 0$, and let $u = \rho(i)$ for each $i \in [k]$. If $u \in [k_M]$, then
		\[
		F_{<u>} = \Lambda(M)_{<u>, \sigma} = M_{<u>, \sigma} = Q_0(i, k).
		\]
		Otherwise, let $r, s \in [k_M]$ be such that $\overline{M}_{<r>, \sigma}$ is properly contained in $M_{<s>, \sigma}$ and
		\[
		\Lambda(M)_{<u>, \sigma} = M_{<r>, \sigma} \cap M_{<s>, \sigma}.
		\]
		
		Since $\Lambda(M)_{<u>, \sigma}$ has $1$'s in columns $i$ and $i+1$ and $0$'s in the remaining columns, both $M_{<r>, \sigma}$ and $M_{<s>, \sigma}$ contain $1$'s in columns $i$ and $i+1$. Since $\overline{M}_{<r>, \sigma}$ is properly contained in $M_{<s>, \sigma}$, if $M_{<s>, \sigma}$ has a $0$ in any column, then $M_{<r>, \sigma}$ has a $1$ in that column, and if $M_{<r>, \sigma}$ has a $0$ in any column, then $M_{<s>, \sigma}$ has a $1$ in that column. Therefore, $M_{<r>, \sigma}$ and $M_{<s>, \sigma}$ do not coincide in any columns other than $i$ and $i+1$. Without loss of generality, suppose that $M_{<r>, \sigma}$ has a $0$ and $M_{<s>, \sigma}$ has a $1$ in column $k+1$.
		
		\textbf{\textit{Claim: }}\textit{$M_{<r>, \sigma}$ has all $0's$ or all $1's$ in the columns different from $i, i+1$, and $k+1$.}
		
		Suppose, for the sake of contradiction, that there is some $0$ and some $1$ in columns other than $i$, $i+1$, and $k+1$. Therefore, there exist $x, y \in [k]$ such that $M_{<r>, \sigma}$ has $0$'s in all the columns in $\{x, x+1, \ldots, y\}$, and $1$'s in columns $x-1$ and $y+1$ (where addition and subtraction are modulo $k$), with $(x-1, y+1) \ne (i+1, i)$. 
		
		Hence, if $\rho' = \langle \rho(x-1), \rho(x), \ldots, \rho(y), s \rangle$, $\sigma' = \langle x-1, x, \ldots, y+1, k+1 \rangle$, and $a'$ is the sequence obtained from $a_{\rho \circ \langle x-1, x, \ldots, y \rangle}$ by appending a $0$, then $\Lambda(M)_{\rho', \sigma'} = a' \miop \MIast{|a'|}$ is a matrix representing the same configuration as some matrix in $\ForbRow$ that is contained in $\Lambda(M)$ as a configuration and has fewer columns than $F$, contradicting the choice of $F$. Therefore, $M_{<r>, \sigma}$ must have all $0$'s or all $1$'s in the columns other than $i$, $i+1$, and $k+1$.
		
	If $M_{<r>, \sigma}$ has $0$'s in the columns other than $i$, $i+1$, and $k+1$, then $M_{<r>, \sigma}$ coincides with the only row of $Q_0(i, k)$. If $M_{<r>, \sigma}$ has $1$'s in the columns other than $i$, $i+1$, and $k+1$, then $M_{<s>, \sigma}$ has $0$'s in the columns other than $i$, $i+1$, and $k+1$. Therefore, $M_{<r>, \sigma}$ and $M_{<s>, \sigma}$ coincide with the rows of $Q_2(i, k)$.
	
	Now assume that $a_i = 1$, and let $u = \rho(i)$ for each $i \in [k]$. If $u \in [k_M]$, then
	\[
	F_{<u>} = \Lambda(M)_{<u>, \sigma} = M_{<u>, \sigma} = Q_1(i, k).
	\]
	Otherwise, let $r, s \in [k_M]$ be such that $\overline{M}_{<r>, \sigma}$ is properly contained in $M_{<s>, \sigma}$ and
	\[
	\Lambda(M)_{<u>, \sigma} = M_{<r>, \sigma} \cap M_{<s>, \sigma}.
	\]
	
	Since $\Lambda(M)_{<u>, \sigma}$ has $0$'s in columns $i$ and $i+1$, and $1$'s in the remaining columns, both $M_{<r>, \sigma}$ and $M_{<s>, \sigma}$ contain $1$'s in the columns other than $i$ and $i+1$, and $M_{<r>, \sigma}$ and $M_{<s>, \sigma}$ do not coincide in columns $i$ and $i+1$. 
	
	Let $(\alpha, \beta)$ and $(\overline{\alpha}, \overline{\beta})$ be the entries of $M_{<r>, \sigma}$ and $M_{<s>, \sigma}$ at columns $i$ and $i+1$, respectively. If $(\alpha, \beta) = (0, 0)$ (respectively, $(\alpha, \beta) = (1, 1)$), then $M_{<r>, \sigma}$ (respectively, $M_{<s>, \sigma}$) coincides with the only row of $Q_1(i, k)$. If $(\alpha, \beta) = (0, 1)$ or $(\alpha, \beta) = (1, 0)$, then $M_{<r>, \sigma}$ and $M_{<s>, \sigma}$ coincide with the rows of $Q_3(i, k)$.
	
	Therefore, for each $i \in [k]$ and $j \in \{0, 1, 2, 3\}$, each row of $Q_j(i, k)$ coincides with some row of $M_{id_{k_M}, \sigma}$. Thus, there exists a quaternary sequence $b = b_1 b_2 \ldots b_k$ such that $R(b)$ is contained in $M$ as a configuration. Since, by \textit{Lemma}~\ref{rb}, $R(b)$ contains some matrix $F' \in \IForbRow$ as a configuration, it follows that $M$ contains $F'$ as a configuration. This contradicts the assumption that $M$ contains no matrix in $\IForbRow$ as a configuration. Therefore, Case 1 cannot occur.
	
	\textbf{Case 2: }Now suppose that $F = a \miop\MVast$ for some $a \in \{0000, 0100, 0010, 1010\}$. By \textit{Lemma }\ref{MVast}, we can say that $F$ has the same configuration as some matrix in\\ $\{\MIV, \overline{\MIV}, \MVast, \overline{\MVast}\}$. 
	
	\textbf{Subcase 1: }Assume that $a_i = 0$, and let $u = \rho(i)$ for each $i \in \{1,3\}$. If $u \in [k_M]$, then
	\[
	F_{<u>} = \Lambda(M)_{<u>, \sigma} = M_{<u>, \sigma} = U_0(i).
	\]
	Otherwise, let $r, s \in [k_M]$ be such that $\overline{M}_{<r>, \sigma}$ is properly contained in $M_{<s>, \sigma}$ and
	\[
	\Lambda(M)_{<u>, \sigma} = M_{<r>, \sigma} \cap M_{<s>, \sigma}.
	\]
	
	Since $\Lambda(M)_{<u>, \sigma}$ has $1$'s in columns $i$ and $i+1$, and $0$'s in the remaining columns, both $M_{<r>, \sigma}$ and $M_{<s>, \sigma}$ contain $1$'s in columns $i$ and $i+1$. Moreover, since $\overline{M}_{<r>, \sigma}$ is properly contained in $M_{<s>, \sigma}$, the rows $M_{<r>, \sigma}$ and $M_{<s>, \sigma}$ do not coincide in any columns other than $i$ and $i+1$. 
	
	Let $(\alpha_1, \alpha_2, \alpha_3, \alpha_4)$ and $(\overline{\alpha}_1, \overline{\alpha}_2, \overline{\alpha}_3, \overline{\alpha}_4)$ be the entries of $M_{<r>, \sigma}$ and $M_{<s>, \sigma}$ at columns $i+2$, $i+3$, $i+4$, and $i+5$, respectively (with all additions involving $i$ taken modulo $6$). Let $\alpha = \alpha_1 \alpha_2 \alpha_3 \alpha_4$. Note that $H_i(\alpha)$ is contained in $\Lambda(M)$ as a configuration. Suppose, for the sake of contradiction, that $\alpha \notin \{0000, 0011, 1100, 1111\}$. By \textit{Lemma}~\ref{X}, $H_i(\alpha)$ contains some matrix $F' \in \ForbRow$ as a configuration that has fewer columns than $F$. Hence, $\Lambda(M)$ contains $F'$ as a configuration with fewer columns than $F$, contradicting the choice of $F$. Therefore, $\alpha \in \{0000, 0011, 1100, 1111\}$. If $\alpha = 0000$ (respectively, $\alpha = 1111$), then $M_{<r>, \sigma}$ (respectively, $M_{<s>, \sigma}$) coincides with the only row of $U_0(i)$. If $\alpha = 0011$ or $\alpha = 1100$, then $M_{<r>, \sigma}$ and $M_{<s>, \sigma}$ coincide with the rows of $U_2(i)$.
	
	Now assume that $a_i = 1$, and let $u = \rho(i)$ for each $i \in \{1,3\}$. If $u \in [k_M]$, then
	\[
	F_{<u>} = \Lambda(M)_{<u>, \sigma} = M_{<u>, \sigma} = U_1(i).
	\]
	Otherwise, let $r, s \in [k_M]$ be such that $\overline{M}_{<r>, \sigma}$ is properly contained in $M_{<s>, \sigma}$ and
	\[
	\Lambda(M)_{<u>, \sigma} = M_{<r>, \sigma} \cap M_{<s>, \sigma}.
	\]
	
	Since $\Lambda(M)_{<u>, \sigma}$ has $0$'s in columns $i$ and $i+1$, and $1$'s in the remaining columns, both $M_{<r>, \sigma}$ and $M_{<s>, \sigma}$ contain $1$'s in the columns other than $i$ and $i+1$. Moreover, since $\overline{M}_{<r>, \sigma}$ is properly contained in $M_{<s>, \sigma}$, the rows $M_{<r>, \sigma}$ and $M_{<s>, \sigma}$ do not coincide in the columns $i$ and $i+1$.
	
	Let $(\alpha, \beta)$ and $(\overline{\alpha}, \overline{\beta})$ be the entries of $M_{<r>, \sigma}$ and $M_{<s>, \sigma}$ at columns $i$ and $i+1$, respectively. If $(\alpha, \beta) = (0,0)$ (respectively, $(\alpha, \beta) = (1, 1)$), then $M_{<r>, \sigma}$ (respectively, $M_{<s>, \sigma}$ ) coincides with the only row of $U_1(i)$. If $(\alpha, \beta) = (0,1)$ or $(\alpha, \beta) = (1,0)$, then $M_{<r>, \sigma}$ and $M_{<s>, \sigma}$ coincide with the rows of $U_3(i)$.
	
	\textbf{Subcase 2: }Assume that $a_2 = 0$, and let $u = \rho(2)$. If $u \in [k_M]$, then
	\[
	F_{<u>} = \Lambda(M)_{<u>, \sigma} = M_{<u>, \sigma} = U_0(2).
	\]
	Otherwise, let $r, s \in [k_M]$ be such that $\overline{M}_{<r>, \sigma}$ is properly contained in $M_{<s>, \sigma}$ and
	\[
	\Lambda(M)_{<u>, \sigma} = M_{<r>, \sigma} \cap M_{<s>, \sigma}.
	\]
	
	Since $\Lambda(M)_{<u>, \sigma}$ has $0$'s in the columns $5$ and $6$, and $1$'s in the remaining columns, both $M_{<r>, \sigma}$ and $M_{<s>, \sigma}$ contain $1$'s in the columns other than $5$ and $6$. Moreover, since $\overline{M}_{<r>, \sigma}$ is properly contained in $M_{<s>, \sigma}$, the rows $M_{<r>, \sigma}$ and $M_{<s>, \sigma}$ do not coincide in the columns $5$ and $6$.
	
	Let $(\alpha, \beta)$ and $(\overline{\alpha}, \overline{\beta})$ be the entries of $M_{<r>, \sigma}$ and $M_{<s>, \sigma}$ at columns $5$ and $6$, respectively. If $(\alpha, \beta) = (0,0)$ (respectively, $(\alpha, \beta) = (1, 1)$), then $M_{<r>, \sigma}$ (respectively, $M_{<s>, \sigma}$ ) coincides with the only row of $U_0(2)$. If $(\alpha, \beta) = (0,1)$ or $(\alpha, \beta) = (1,0)$, then $M_{<r>, \sigma}$ and $M_{<s>, \sigma}$ coincide with the rows of $U_3(2)$.
	
	Now assume that $a_2 = 1$, and let $u = \rho(2)$. If $u \in [k_M]$, then
	\[
	F_{<u>} = \Lambda(M)_{<u>, \sigma} = M_{<u>, \sigma} = U_1(2).
	\]
	Otherwise, let $r, s \in [k_M]$ be such that $\overline{M}_{<r>, \sigma}$ is properly contained in $M_{<s>, \sigma}$ and
	\[
	\Lambda(M)_{<u>, \sigma} = M_{<r>, \sigma} \cap M_{<s>, \sigma}.
	\]
	
	Since $\Lambda(M)_{<u>, \sigma}$ has $1$'s in columns $5$ and $6$, and $0$'s in the remaining columns, both $M_{<r>, \sigma}$ and $M_{<s>, \sigma}$ contain $1$'s in the columns $5$ and $6$. Moreover, since $\overline{M}_{<r>, \sigma}$ is properly contained in $M_{<s>, \sigma}$, the rows $M_{<r>, \sigma}$ and $M_{<s>, \sigma}$ do not coincide in the columns $1$, $2$, $3$ and $4$.
	
	Let $(\alpha_1, \alpha_2, \alpha_3, \alpha_4)$ and $(\overline{\alpha}_1, \overline{\alpha}_2, \overline{\alpha}_3, \overline{\alpha}_4)$ be the entries of $M_{<r>, \sigma}$ and $M_{<s>, \sigma}$ at columns $1$, $2$, $3$, and $4$, respectively. Let $\alpha = \alpha_1 \alpha_2 \alpha_3 \alpha_4$. Note that $H_2(\alpha)$ is contained in $\Lambda(M)$ as a configuration. Based on the argument of the previous subcase, we have $\alpha = \{0000, 0011, 1100, 1111\}$. If $\alpha = 0000$ (respectively, $\alpha = 1111$), then $M_{<r>, \sigma}$ (respectively, $M_{<s>, \sigma}$) coincides with the only row of $U_1(2)$. If $\alpha = 0011$ or $\alpha = 1100$, then $M_{<r>, \sigma}$ and $M_{<s>, \sigma}$ coincide with the rows of $U_2(2)$.
	
	\textbf{Subcase 3: }By construction, $a_4 = 0$. Let $u = \rho(4)$. If $u \in [k_M]$, then
		\[
	F_{<u>} = \Lambda(M)_{<u>, \sigma} = M_{<u>, \sigma} = U_0(4).
	\]
	Otherwise, let $r, s \in [k_M]$ be such that $\overline{M}_{<r>, \sigma}$ is properly contained in $M_{<s>, \sigma}$ and
	\[
	\Lambda(M)_{<u>, \sigma} = M_{<r>, \sigma} \cap M_{<s>, \sigma}.
	\]
	
	Since $\Lambda(M)_{<u>, \sigma}$ has $1$'s in columns $1$, $4$ and $5$, and $0$'s in the remaining columns, both $M_{<r>, \sigma}$ and $M_{<s>, \sigma}$ contain $1$'s in the columns $1$, $4$ and $5$. Moreover, since $\overline{M}_{<r>, \sigma}$ is properly contained in $M_{<s>, \sigma}$, the rows $M_{<r>, \sigma}$ and $M_{<s>, \sigma}$ do not coincide in the columns $2$, $3$ and $6$.
	
	Let $(\gamma_1, \gamma_2, \gamma_3)$ and $(\overline{\gamma}_1, \overline{\gamma}_2, \overline{\gamma}_3)$ be the entries of $M_{<r>, \sigma}$ and $M_{<s>, \sigma}$ at columns $2$, $3$, and $6$, respectively. Let $\gamma = \gamma_1 \gamma_2 \gamma_3$. Note that $G(\gamma)$ is contained in $\Lambda(M)$ as a configuration. Suppose, for the sake of contradiction, that $\gamma$ is nonconstant. By \textit{Lemma}~\ref{G}, $G(\gamma)$ contains some matrix $F' \in \ForbRow$ as a configuration that has fewer columns than $F$. Hence, $\Lambda(M)$ contains $F'$ as a configuration with fewer columns than $F$, contradicting the choice of $F$. Therefore, $\gamma$ must be constant. If $\gamma = 000$ (respectively, $\gamma = 111$), then $M_{<r>, \sigma}$ (respectively, $M_{<s>, \sigma}$) coincides with the only row of $U_0(4)$.
	
	Therefore, for each $i \in [4]$ and $j \in \{0, 1, 2, 3\}$, each row of $U_j(i)$ coincides with some row of $M_{id_{k_M}, \sigma}$. Thus, there exists a quaternary sequence $b = b_1 b_2 b_3 b_4$ such that $W(b)$ is contained in $M$ as a configuration. Since, by \textit{Lemma}~\ref{W}, $W(b)$ contains some matrix $F' \in \IForbRow$ as a configuration, it follows that $M$ contains $F'$ as a configuration. This contradicts the assumption that $M$ contains no matrix in $\IForbRow$ as a configuration. Therefore, Case 2 cannot occur.

	The fact that Case 1 and Case 2 cannot occur contradicts $F \in \ForbRow$. This contradiction completes the proof of the Theorem.
		\end{proof}
		\section{Forbidden Induced Subgraphs of Semi-Transitive Split Graphs}
		In this section, we establish the connection between the $I$-circular property of matrices and semi-transitive split graphs, thereby characterizing the latter through forbidden induced subgraphs. The following theorem follows from the proofs of \textit{Theorem 4} and \textit{Theorem 5} in \cite{kitaev2024semi}.
		
		\begin{theorem}\label{sgicp}
			A split graph $G$ is semi-transitive if and only if the matrix $A(G)$ has $I$-circular property.
		\end{theorem}
\begin{figure}[h]
\begin{minipage}{0.32\textwidth}
	\centering
	\begin{tikzpicture}[line cap=round,line join=round,>=triangle 45,x=1cm,y=1cm, scale=2]
		\begin{scriptsize}
			% vertices
			\node[draw, circle, fill = qqzzcc] (k-1) at (1.2,3.25) {};
			\draw[color=qqzzcc] (1,3.4) node {k-1};
			\node[draw, circle, fill = qqzzcc] (3) at (1.8,3.25) {};
			\draw[color=qqzzcc] (1.88,3.4) node {3};
			\node[draw, circle, fill = qqzzcc] (2) at (2.0,2.75) {};
			\draw[color=qqzzcc] (2.1532493492248213,2.720110568527594) node {2};
			\node[draw, circle, fill = qqzzcc] (1) at (1.5,2.25) {};
			\draw[color=qqzzcc] (1.49,2.56) node {1};
			\node[draw, circle, fill = qqzzcc] (k) at (1.0,2.75) {};
			\draw[color=qqzzcc] (0.85978665910432,2.730137411086668) node {k};
			\node[draw, circle, fill = qqccqq] (a) at (0.7,3.1) {};
			\node[draw, circle, fill = qqccqq] (b) at (2.27,3.1) {};
			\node[draw, circle, fill = qqccqq] (c) at (2,2.3) {};
			\node[draw, circle, fill = qqccqq] (d) at (1,2.3) {};
			\draw[color=black] (1.5,3.4) node {...};
			\node[draw, circle, fill = qqzzcc] (k+1) at (1.5,1.9) {};
			\draw[color=qqzzcc] (1.72,1.8) node {k+1};
			
			% edges
			\draw (k-1) -- (3);
			\draw (2) -- (3);
			\draw (1) -- (3);
			\draw (k) -- (3);
			\draw (k+1) -- (3);
			\draw (k-1) -- (1);
			\draw (2) -- (1);
			\draw (k+1) -- (1);
			\draw (k) -- (1);
			\draw (k-1) -- (2);
			\draw (k+1) -- (2);
			\draw (k) -- (2);
			\draw (k-1) -- (k+1);
			\draw (k) -- (k+1);
			\draw (k-1) -- (k);
			\draw (a) -- (k);
			\draw (a) -- (k-1);
			\draw (3) -- (b);
			\draw (2) -- (b);
			\draw (1) -- (c);
			\draw (2) -- (c);
			\draw (k) -- (d);
			\draw (1) -- (d);
		\end{scriptsize}
	\end{tikzpicture}\\{$\MIast{k}$ for $k \geq 3$}
\end{minipage}
\hfill
\begin{minipage}{0.32\textwidth}
	\centering
\begin{tikzpicture}[line cap=round,line join=round,>=triangle 45,x=1cm,y=1cm, scale=2]
	\begin{scriptsize}
		% vertices
		\node[draw, circle, fill = qqzzcc] (k-1) at (1.2,3.25) {};
		\draw[color=qqzzcc] (1,3.4) node {k-2};
		\node[draw, circle, fill = qqzzcc] (3) at (1.8,3.25) {};
		\draw[color=qqzzcc] (1.88,3.4) node {3};
		\node[draw, circle, fill = qqzzcc] (2) at (2.0,2.75) {};
		\draw[color=qqzzcc] (2.15,2.72) node {2};
		\node[draw, circle, fill = qqzzcc] (1) at (1.8,2.25) {};
		\draw[color=qqzzcc] (1.88,2.1) node {1};
		\node[draw, circle, fill = qqzzcc] (k) at (1.0,2.75) {};
		\draw[color=qqzzcc] (0.75,2.73) node {k - 1};
		\node[draw, circle, fill = qqccqq] (a) at (0.7,3.1) {};
		\node[draw, circle, fill = qqccqq] (b) at (2.27,3.1) {};
		\node[draw, circle, fill = qqccqq] (c) at (2.25,2.35) {};
		\node[draw, circle, fill = qqccqq] (d) at (1.5,1.9) {};
		\node[draw, circle, fill = qqccqq] (e) at (0.65,2.35) {};
		\draw[color=black] (1.5,3.4) node {...};
		\node[draw, circle, fill = qqzzcc] (k+1) at (1.2,2.25) {};
		\draw[color=qqzzcc] (1.1,2.1) node {k};
		
		% edges
		\draw (k-1) -- (3);
		\draw (2) -- (3);
		\draw (1) -- (3);
		\draw (k) -- (3);
		\draw (k+1) -- (3);
		\draw (k-1) -- (1);
		\draw (2) -- (1);
		\draw (k+1) -- (1);
		\draw (k) -- (1);
		\draw (k-1) -- (2);
		\draw (k+1) -- (2);
		\draw (k) -- (2);
		\draw (k-1) -- (k+1);
		\draw (k) -- (k+1);
		\draw (k-1) -- (k);
		\draw (a) -- (k);
		\draw (a) -- (k-1);
		\draw (3) -- (b);
		\draw (2) -- (b);
		\draw (1) -- (c);
		\draw (2) -- (c);
		\draw (k+1) -- (d);
		\draw (1) -- (d);
		\draw (k-1) -- (d);
		\draw (2) -- (d);
		\draw (3) -- (d);
		\draw (k) -- (e);
		\draw (k+1) -- (e);
		\draw (2) -- (e);
		\draw (3) -- (e);
		\draw (k-1) -- (e);
	\end{scriptsize}
\end{tikzpicture}\\{$\MII{k}$ for $k \geq 4$}
\end{minipage}
\hfill
\begin{minipage}{0.32\textwidth}
	\centering
\begin{tikzpicture}[line cap=round,line join=round,>=triangle 45,x=1cm,y=1cm, scale=2]
	\begin{scriptsize}
		% vertices
		\node[draw, circle, fill = qqzzcc] (k-1) at (1.2,3.25) {};
		\draw[color=qqzzcc] (1.1,3.4) node {k-1};
		\node[draw, circle, fill = qqzzcc] (3) at (1.8,3.25) {};
		\draw[color=qqzzcc] (1.88,3.4) node {3};
		\node[draw, circle, fill = qqzzcc] (2) at (2.0,2.75) {};
		\draw[color=qqzzcc] (2.15,2.72) node {2};
		\node[draw, circle, fill = qqzzcc] (1) at (1.8,2.25) {};
		\draw[color=qqzzcc] (1.88,2.1) node {1};
		\node[draw, circle, fill = qqzzcc] (k) at (1.0,2.75) {};
		\draw[color=qqzzcc] (0.8,2.73) node {k};
		\node[draw, circle, fill = qqccqq] (a) at (0.7,3.1) {};
		\node[draw, circle, fill = qqccqq] (b) at (2.27,3.1) {};
		\node[draw, circle, fill = qqccqq] (c) at (2.15,2.35) {};
		\node[draw, circle, fill = qqccqq] (d) at (1.5,1.9) {};
		\draw[color=black] (1.5,3.4) node {...};
		\node[draw, circle, fill = qqzzcc] (k+1) at (1.2,2.25) {};
		\draw[color=qqzzcc] (1,2.1) node {k + 1};
		
		% edges
		\draw (k-1) -- (3);
		\draw (2) -- (3);
		\draw (1) -- (3);
		\draw (k) -- (3);
		\draw (k+1) -- (3);
		\draw (k-1) -- (1);
		\draw (2) -- (1);
		\draw (k+1) -- (1);
		\draw (k) -- (1);
		\draw (k-1) -- (2);
		\draw (k+1) -- (2);
		\draw (k) -- (2);
		\draw (k-1) -- (k+1);
		\draw (k) -- (k+1);
		\draw (k-1) -- (k);
		\draw (a) -- (k);
		\draw (a) -- (k-1);
		\draw (3) -- (b);
		\draw (2) -- (b);
		\draw (1) -- (c);
		\draw (2) -- (c);
		\draw (k+1) -- (d);
		\draw (k-1) -- (d);
		\draw (2) -- (d);
		\draw (3) -- (d);
	\end{scriptsize}
\end{tikzpicture}\\{$\MIII{k}$ for $k \geq 3$}
\end{minipage}

\vspace{0.5cm}
\noindent
\begin{minipage}{0.24\textwidth}
	\centering
\begin{tikzpicture}[line cap=round,line join=round,>=triangle 45,x=1cm,y=1cm, scale=1.75]
	\begin{scriptsize}
		% vertices
		\node[draw, circle, fill = qqzzcc] (4) at (1.2,3.25) {};
		\draw[color=qqzzcc] (1,3.4) node {4};
		\node[draw, circle, fill = qqzzcc] (3) at (1.8,3.25) {};
		\draw[color=qqzzcc] (1.88,3.4) node {3};
		\node[draw, circle, fill = qqzzcc] (2) at (2.05,2.75) {};
		\draw[color=qqzzcc] (2.2,2.72) node {2};
		\node[draw, circle, fill = qqzzcc] (1) at (1.8,2.25) {};
		\draw[color=qqzzcc] (1.88,2.1) node {1};
		\node[draw, circle, fill = qqzzcc] (5) at (0.95,2.75) {};
		\draw[color=qqzzcc] (0.75,2.73) node {5};
		\node[draw, circle, fill = qqccqq] (b) at (1.5,2) {};
		\node[draw, circle, fill = qqccqq] (c) at (2.3,2.25) {};
		\node[draw, circle, fill = qqccqq] (d) at (1.5,3.7) {};
		\node[draw, circle, fill = qqccqq] (e) at (0.7,2.25) {};
		\node[draw, circle, fill = qqzzcc] (6) at (1.2,2.25) {};
		\draw[color=qqzzcc] (1.1,2.1) node {6};
		
		% edges
		\draw (4) -- (3);
		\draw (2) -- (3);
		\draw (1) -- (3);
		\draw (5) -- (3);
		\draw (6) -- (3);
		\draw (4) -- (1);
		\draw (2) -- (1);
		\draw (6) -- (1);
		\draw (5) -- (1);
		\draw (4) -- (2);
		\draw (6) -- (2);
		\draw (5) -- (2);
		\draw (4) -- (6);
		\draw (5) -- (6);
		\draw (4) -- (5);
		\draw (4) -- (d);
		\draw (3) -- (d);
		\draw (5) -- (e);
		\draw (6) -- (e);
		\draw (1) -- (c);
		\draw (2) -- (c);
		\draw (4) -- (b);
		\draw (6) -- (b);
		\draw (2) -- (b);
	\end{scriptsize}
\end{tikzpicture}\\{$\MIV$}
\end{minipage}
\hfill
\begin{minipage}{0.24\textwidth}
	\centering
\begin{tikzpicture}[line cap=round,line join=round,>=triangle 45,x=1cm,y=1cm, scale=2]
	\begin{scriptsize}
		% vertices
		\node[draw, circle, fill = qqzzcc] (4) at (1.2,3.25) {};
		\draw[color=qqzzcc] (1,3.4) node {4};
		\node[draw, circle, fill = qqzzcc] (3) at (1.8,3.25) {};
		\draw[color=qqzzcc] (1.88,3.4) node {3};
		\node[draw, circle, fill = qqzzcc] (2) at (2.0,2.75) {};
		\draw[color=qqzzcc] (2.15,2.72) node {2};
		\node[draw, circle, fill = qqzzcc] (1) at (1.5,2.25) {};
		\draw[color=qqzzcc] (1.49,2.56) node {1};
		\node[draw, circle, fill = qqzzcc] (5) at (1.0,2.75) {};
		\draw[color=qqzzcc] (0.75,2.73) node {5};
		\node[draw, circle, fill = qqccqq] (e) at (1.5,1.9) {};
		\node[draw, circle, fill = qqccqq] (c) at (2,2.3) {};
		\node[draw, circle, fill = qqccqq] (d) at (1.5,3.7) {};
		\node[draw, circle, fill = qqccqq] (b) at (1,2.3) {};
		
		% edges
		\draw (4) -- (3);
		\draw (2) -- (3);
		\draw (1) -- (3);
		\draw (5) -- (3);
		\draw (4) -- (1);
		\draw (2) -- (1);
		\draw (5) -- (1);
		\draw (4) -- (2);
		\draw (5) -- (2);
		\draw (4) -- (5);
		\draw (4) -- (d);
		\draw (3) -- (d);
		\draw (1) -- (e);
		\draw (2) -- (e);
		\draw (3) -- (e);
		\draw (4) -- (e);
		\draw (1) -- (c);
		\draw (2) -- (c);
		\draw (4) -- (b);
		\draw (5) -- (b);
		\draw (1) -- (b);
	\end{scriptsize}
\end{tikzpicture}\\{$\MV$}
\end{minipage}
\hfill
\begin{minipage}{0.24\textwidth}
	\centering
\begin{tikzpicture}[line cap=round,line join=round,>=triangle 45,x=1cm,y=1cm, scale=2]
	\begin{scriptsize}
		% vertices
		\node[draw, circle, fill = qqzzcc] (4) at (1.5,2.5) {};
		\draw[color=qqzzcc] (1.6,2.65) node {4};
		\node[draw, circle, fill = qqzzcc] (2) at (1.95,2.75) {};
		\draw[color=qqzzcc] (2.15,2.72) node {2};
		\node[draw, circle, fill = qqzzcc] (1) at (1.5,2.15) {};
		\draw[color=qqzzcc] (1.49,1.95) node {1};
		\node[draw, circle, fill = qqzzcc] (5) at (1.05,2.75) {};
		\draw[color=qqzzcc] (0.77,2.73) node {3};
		\node[draw, circle, fill = qqccqq] (c) at (2.25,2.15) {};
		\node[draw, circle, fill = qqccqq] (d) at (1.5,3.6) {};
		\node[draw, circle, fill = qqccqq] (b) at (0.75,2.15) {};
		
		% edges
		\draw (4) -- (1);
		\draw (2) -- (1);
		\draw (5) -- (1);
		\draw (4) -- (2);
		\draw (5) -- (2);
		\draw (4) -- (5);
		\draw (5) -- (d);
		\draw (2) -- (d);
		\draw (4) -- (d);
		\draw (4) -- (c);
		\draw (4) -- (b);
		\draw (1) -- (c);
		\draw (2) -- (c);
		\draw (5) -- (b);
		\draw (1) -- (b);
	\end{scriptsize}
\end{tikzpicture}\\{$\MVI$}
\end{minipage}
\hfill
\begin{minipage}{0.24\textwidth}
	\centering
	\begin{tikzpicture}[line cap=round,line join=round,>=triangle 45,x=1cm,y=1cm, scale=2]
		\begin{scriptsize}
			% vertices
			\node[draw, circle, fill = qqzzcc] (4) at (1.2,3.25) {};
			\draw[color=qqzzcc] (1,3.4) node {4};
			\node[draw, circle, fill = qqzzcc] (3) at (1.8,3.25) {};
			\draw[color=qqzzcc] (1.88,3.4) node {3};
			\node[draw, circle, fill = qqzzcc] (2) at (2.0,2.75) {};
			\draw[color=qqzzcc] (2.15,2.72) node {2};
			\node[draw, circle, fill = qqzzcc] (1) at (1.5,2.25) {};
			\draw[color=qqzzcc] (1.49,2.56) node {1};
			\node[draw, circle, fill = qqzzcc] (5) at (1.0,2.75) {};
			\draw[color=qqzzcc] (0.75,2.73) node {5};
			\node[draw, circle, fill = qqccqq] (e) at (1.5,1.95) {};
			\node[draw, circle, fill = qqccqq] (c) at (2,2.3) {};
			\node[draw, circle, fill = qqccqq] (d) at (1.5,3.7) {};
			\node[draw, circle, fill = qqccqq] (b) at (1,2.3) {};
			
			% edges
			\draw (4) -- (3);
			\draw (2) -- (3);
			\draw (1) -- (3);
			\draw (5) -- (3);
			\draw (4) -- (1);
			\draw (2) -- (1);
			\draw (5) -- (1);
			\draw (4) -- (2);
			\draw (5) -- (2);
			\draw (4) -- (5);
			\draw (4) -- (d);
			\draw (3) -- (d);
			\draw (2) -- (e);
			\draw (3) -- (e);
			\draw (5) -- (e);
			\draw (1) -- (c);
			\draw (2) -- (c);
			\draw (4) -- (b);
			\draw (5) -- (b);
			\draw (1) -- (b);
		\end{scriptsize}
	\end{tikzpicture}\\{$\MVII$}
\end{minipage}
\vspace{0.25cm}\caption{Minimal Forbidden Induced Subgraphs for Semi-Transitive Spit Graphs $\GForbRow$.}
\label{gsplit}
\end{figure}

For a $(0,1)$-matrix $M$ of size $m \times n$, we define $SG(M)$ as the split graph whose independent set corresponds to the rows of $M$, clique corresponds to its columns, and where a vertex $i$ from the independent set is adjacent to a vertex $j$ from the clique if and only if $m_{ij} = 1$.

\begin{lemma}\label{gsp}
	If $F \in \IForbRow$, then $SG(F)$ contains an induced subgraph isomorphic to some graph in $\GForbRow$.
\end{lemma}
\begin{proof}
	Let $F$ be a matrix in $\IForbRow$. If $F \in \{\MIast{k}, \MIII{k}, \MII{k+1}\colon k \geq 3\} \cup \{0101 \miop \MIast{4},\, \MIV,\, \MV,\, \MVI\}$, then $SG(F)$ is isomorphic to some graph in $\GForbRow$. If $F = 0100 \miop \MII{4}$, then $SG(F_{\mathrm{id}_4,\langle 1,2,4\rangle})$ is isomorphic to $\MIast{3}$. This completes the proof of the lemma.
\end{proof}

\begin{lemma}\label{nwr}
	All graphs in $\GForbRow$ are minimal non-semi-transitive.
\end{lemma}
\begin{proof}
	Let $G$ be a graph in $\GForbRow$. Observe that $A(G) \in \IForbRow$. Therefore, by \textit{Corollary~\ref{NFIM}}, \textit{Lemma~\ref{L1}}, and \textit{Theorem~\ref{sgicp}}, $G$ is a minimal non-semi-transitive graph.
\end{proof}
\begin{theorem}
	Let $G$ be a split graph. Then, $G$ is semi-transitive if and only if $G$ contains none of the graphs in $\GForbRow$ (depicted in Figure~\ref{gsplit}) as induced subgraph.
\end{theorem}
\begin{proof}
	Assume that $G$ is semi-transitive. Suppose, for the sake of contradiction, that $G$ contains some graph in $\GForbRow$ as an induced subgraph. Then, by \textit{Lemma~\ref{nwr}} and the hereditary nature of semi-transitive graphs, $G$ would be non-semi-transitive, which is a contradiction.
	
	Suppose that $G$ contains none of the graphs in $\GForbRow$ as an induced subgraph. Assume, for the sake of contradiction, that $G$ is not semi-transitive. Then, by \textit{Theorem~\ref{sgicp}}, $A(G)$ does not have the $I$-circular property. Moreover, by \textit{Theorem~\ref{icp}}, $A(G)$ contains some matrix in $\IForbRow$ as a configuration. Therefore, by \textit{Lemma~\ref{gsp}}, $G$ contains an induced subgraph isomorphic to some graph in $\GForbRow$, which is a contradiction. This concludes the proof.
\end{proof}

\bibliographystyle{plain} % We choose the "plain" reference style
\bibliography{ref1} % Entries are in the refs.bib file

@book{kitaev2015words,
	author = {S. Kitaev and V. Lozin},
	title = {Words and Graphs},
	publisher = {Springer},
	year = {2015}
}

@book{MR1367739,
	author = {D. B. West},
	title = {Introduction to Graph Theory},
	publisher = {Prentice Hall},
	year = {1996}
}

@article{corneil1981,
	author = {D. G. Corneil and H. Lerchs and L. S. Burlingham},
	title = {Complement reducible graphs},
	journal = {Discrete Appl. Math.},
	volume = {3},
	pages = {163--174},
	year = {1981}
}

@misc{broere2018word,
	author = {B. Broere},
	title = {Word-representable graphs},
	note = {Master's thesis, Radboud Univ. Nijmegen},
	year = {2018}
}

@article{article,
	author = {V. B. Le},
	title = {On opposition graphs, coalition graphs, and bipartite permutation graphs},
	journal = {Discrete Appl. Math.},
	volume = {168},
	pages = {26--33},
	year = {2014}
}

@article{HALLDORSSON2016164,
	author = {M. M. Halld\'{o}rsson and S. Kitaev and A. Pyatkin},
	title = {Semi-transitive orientations and word-representable graphs},
	journal = {Discrete Appl. Math.},
	volume = {201},
	pages = {164--171},
	year = {2016}
}

@article{huang2024embedding,
	author = {S. Huang and S. Kitaev and A. Pyatkin},
	title = {An embedding technique in the study of word-representability of graphs},
	journal = {Discrete Appl. Math.},
	volume = {346},
	pages = {170--182},
	year = {2024}
}

@article{iamthong2021semi,
	author = {K. Iamthong and S. Kitaev},
	title = {Semi-transitivity of directed split graphs generated by morphisms},
	journal = {J. Combin.},
	volume = {14},
	number = {1},
	year = {2023}
}

@article{iamthong2022word,
	author = {K. Iamthong},
	title = {Word-representability of split graphs generated by morphisms},
	journal = {Discrete Appl. Math.},
	volume = {314},
	pages = {284--303},
	year = {2022}
}

@article{srinivasan2024minimum,
	author = {S. Eshwar and H. Ramesh},
	title = {Minimum length word-representants of word-representable graphs},
	journal = {Discrete Appl. Math.},
	volume = {343},
	pages = {149--158},
	year = {2024}
}

@article{booth1976pq,
	author = {K. S. Booth and G. S. Lueker},
	title = {Testing for the consecutive ones property, interval graphs, and graph planarity using PQ-tree algorithms},
	journal = {J. Comput. System Sci.},
	volume = {13},
	number = {3},
	pages = {335--379},
	year = {1976}
}

@article{mdsafe1,
	author = {M. D. Safe},
	title = {Characterization and linear-time detection of minimal obstructions to concave-round graphs and the circular-ones property},
	journal = {J. Graph Theory},
	volume = {93},
	pages = {268--298},
	year = {2020}
}

@article{MR1857399,
	author = {J. Sawada},
	title = {Generating bracelets in constant amortized time},
	journal = {SIAM J. Comput.},
	volume = {31},
	number = {1},
	pages = {259--268},
	year = {2001}
}

@book{golumbic1980,
	author = {M. C. Golumbic},
	title = {Algorithmic Graph Theory and Perfect Graphs},
	publisher = {Academic Press},
	year = {1980}
}

@article{kitaev2008representable,
	author = {S. Kitaev and A. Pyatkin},
	title = {On representable graphs},
	journal = {J. Autom. Lang. Comb.},
	volume = {13},
	number = {1},
	pages = {45--54},
	year = {2008}
}

@article{kitaev13,
	author = {S. Kitaev},
	title = {On graphs with representation number 3},
	journal = {J. Autom. Lang. Comb.},
	volume = {18},
	number = {2},
	pages = {97--112},
	year = {2013}
}

@article{kitaev2008word,
	author = {S. Kitaev and S. Seif},
	title = {Word problem of the Perkins semigroup via directed acyclic graphs},
	journal = {Order},
	volume = {25},
	number = {3},
	pages = {177--194},
	year = {2008}
}

@inproceedings{kitaev2017comprehensive,
	author = {S. Kitaev},
	title = {A comprehensive introduction to the theory of word-representable graphs},
	booktitle={Developments in Language Theory: 21st International Conference, DLT 2017, Li{\`e}ge, Belgium, August 7-11, 2017, Proceedings},
	pages = {36--67},
	year = {2017},
	publisher = {Springer}
}

@article{kitaev2021human,
	author = {S. Kitaev and H. Sun},
	title = {Human-verifiable proofs in the theory of word-representable graphs},
	journal = {RAIRO Theor. Inform. Appl.},
	volume = {58},
	pages = {9},
	year = {2024}
}

@inproceedings{kitaev2011representability,
	author = {S. Kitaev and P. Salimov and C. Severs and H. {\'U}lfarsson},
	title = {On the representability of line graphs},
	booktitle={Developments in Language Theory: 15th International Conference, DLT 2011, Milan, Italy, July 19-22, 2011. Proceedings 15},
	pages = {478--479},
	year = {2011},
	publisher = {Springer}
}

@article{kitaev2021word,
	author = {S. Kitaev and Y. Long and J. Ma and H. Wu},
	title = {Word-representability of split graphs},
	journal = {J. Combin.},
	volume = {12},
	number = {4},
	pages = {725--746},
	year = {2021}
}

@article{chen2022representing,
	author = {H. Z. Q. Chen and S. Kitaev and A. Saito},
	title = {Representing split graphs by words},
	journal = {Discuss. Math. Graph Theory},
	volume = {42},
	number = {4},
	pages = {1263--1280},
	year = {2022}
}

@article{kitaev2024semi,
	author = {S. Kitaev and A. Pyatkin},
	title = {On semi-transitive orientability of split graphs},
	journal = {Inform. Process. Lett.},
	volume = {184},
	pages = {106435},
	year = {2024}
}

@article{roy2025word,
	author = {R. Suchanda and H. Ramesh},
	title = {Word-representability of split graphs with independent set of size 4},
	journal = {arXiv preprint arXiv:2507.08483},
	year = {2025}
}

@article{tucker1972structure,
	author = {A. C. Tucker},
	title = {A structure theorem for the consecutive 1's property},
	journal = {J. Combin. Theory Ser. B},
	volume = {12},
	number = {2},
	pages = {153--162},
	year = {1972}
}

@article{mdsafe2,
	author = {M. D. Safe},
	title = {Circularly compatible ones, D-circularity, and proper circular-arc bigraphs},
	journal = {SIAM J. Discrete Math.},
	volume = {35},
	number = {2},
	pages = {707--751},
	year = {2021}
}
\end{document}